\documentclass{amsart}
\usepackage{jswmacros, amssymb, xypic}
\input xy

\newcommand{\LL}{\mathcal{L}}
\newcommand{\FS}{S}
\newcommand{\CF}{\mathcal{F}}
\DeclareMathOperator{\GU}{GU}

\numberwithin{equation}{subsection}
\numberwithin{Theorem}{subsection}

\title{The local Jacquet-Langlands correspondence via Fourier analysis}
\email{jared@math.ucla.edu}
\address{UCLA Mathematics Department \\ Box 951555 \\ Los Angeles, CA 90095-1555}
\author{Jared Weinstein}

\begin{document}

\begin{abstract}
Let $F$ be a locally compact non-Archimedean field, and let $B/F$ be a division algebra of dimension 4.   The Jacquet-Langlands correspondence provides a bijection between smooth irreducible representations $\pi'$ of $B^\times$ of dimension $>1$ and irreducible cuspidal representations of $\GL_2(F)$.  We present a new construction of this bijection in which the preservation of epsilon factors is automatic.  This is done by constructing a family of pairs $(\mathcal{L},\rho)$, where $\mathcal{L}\subset M_2(F)\times B$ is an order and $\rho$ is a finite-dimensional representation of a certain subgroup of $\GL_2(F)\times B^\times$ containing $\mathcal{L}^\times$.    Let $\pi\otimes\pi'$ be an irreducible representation of $\GL_2(F)\times B^{\times}$;  we show that $\pi\otimes\pi'$ contains such a $\rho$ if and only if $\pi$ is cuspidal and corresponds to $\check{\pi}'$ under Jacquet-Langlands, and also that every $\pi$ and $\pi'$ arises this way.    The agreement of epsilon factors is reduced to a Fourier-analytic calculation on a finite ring quotient of $\mathcal{L}$.
\end{abstract}

\maketitle

\large

\section{Introduction}

Let $F$ be a non-Archimedean local field,  {\em{i.e.}} a finite extension either of $\Q_p$ or of the field of Laurent series over the finite field $\F_p$.  Let $B/F$ be a central simple algebra of
dimension $n^2$. The Jacquet-Langlands correspondence
assigns to each irreducible admissible representation $\pi'$ of
$B^{\times}$ a square-integrable representation $\pi$ of $\GL_n(F)$. The
passage $\pi'\mapsto\pi$ is characterized by a character relation;
it also manifests as a relationship between epsilon factors.  This
reciprocity between $\GL_n(F)$ and $B^{\times}$ was first proven in the
case of $n=2$ by Jacquet and Langlands~\cite{JacquetLanglands} in
both the local and global settings.   In the case of a division algebra in characteristic 0 it was established for all $n$
by Rogawski~\cite{Rogawski}.  The case of a general twist of $\GL_n(F)$ was carried out by Deligne, Kazhdan and Vign\'eras in~\cite{DKV} in characteristic 0 and Badulescu~\cite{badulescu} in characteristic $p$.  Each of these cases was accomplished by embedding the local problem into a global one and then applying trace formula methods.

Since then there has been a great deal of effort to construct the
Jacquet-Langlands correspondence in an explicit manner using purely
local techniques.  The simplest case is when $\pi'$ and $\pi$ are
both associated to a so-called ``admissible pair" $(E,\theta)$,
where $E/F$ is a field extension of degree $n$ and $\theta$ is a
character of $E^{\times}$. (All supercuspidal $\pi$ will arise this way if
$p\nmid n$.)  In this case the corresponding $\pi$ was constructed
explicitly by Howe~\cite{Howe}; G\'erardin~\cite{Gerardin}
constructed the representation $\pi'$ and proved that the epsilon
factors of $\pi$ and $\pi'$ agree. Henniart~\cite{HenniartJLI}
showed that if $n$ is a prime distinct from $p$ the representations $\pi$ and
$\pi'$ so constructed have the correct character identity.  Using
the technology of types laid down by Bushnell and Kutzko
in~\cite{BushnellKutzko}, Henniart and Bushnell construct the
explicit correspondence in the case of $n=p$ in ~\cite{HenniartJLII}
and in the case of $n$ a power of $p$ in~\cite{HenniartJLIII}.

In this paper we present a novel approach to the passage
$\pi'\mapsto\pi$ in the case $n=2$ in such a way that the
preservation of epsilon factors is manifest in the construction. Our
approach is entirely Fourier-analytic, and there is no special
treatment needed for the case $p=2$.   In that sense it is similar to G\'erardin-Li\~cite{GerardinLi}.
Unlike that paper, however, our method is linked to the theory of strata
developed for $\GL_n$ in~\cite{BushnellKutzko}.  The theory is summarized in Section~\ref{chainorders}.   Roughly speaking, a
stratum for $\GL_2$ is a certain sort of character of a compact open
subgroup of $\GL_2(F)$. Then irreducible representations of
$\GL_2(F)$ can be conveniently classified according to which strata
they contain. There is a notion of simple stratum:  these are parametrized by certain regular elliptic elements $\beta\in \GL_2(F)$.   It can be shown that (up to a twist) an admissible representation of
$\GL_2(F)$ contains a simple stratum if and only if it is
supercuspidal.  A similar notion of stratum exists for $B^{\times}$, and
strata for $B^{\times}$ are easily seen to be more or less the same objects
as simple strata for $\GL_2(F)$. It is therefore natural to try to
define the correspondence $\pi'\mapsto\pi$ relative to each stratum.

Let $\FS$ be a simple stratum associated to the regular elliptic element $\beta\in\GL_2(F)$, and let $\FS'$ be the stratum in $B^\times$ corresponding to $\FS$.   We choose an embedding of the field $E=F(\beta)$ into $B$.  Let $\Delta\from E\to M_2(F)\times B$ be the diagonal map.  We construct what we have called a
``linking order" $\LL_{\FS}$ inside $M_2(F)\times B$;  this is a $\Delta(\OO_E)$-order defined by certain congruence conditions.  We
then define a irreducible (and thus finite-dimensional) representation $\rho_{\FS}$ of the unit group $\LL_{\FS}^{\times}$ which is trivial on $\Delta(\OO_F^\times)$.   Then loosely speaking, the induction of $\rho_\FS$ to $\GL_2(F)\times B^\times$ will realize the Jacquet-Langlands correspondence for those representations $\pi$ which contain $\FS$.

To make this precise, we must pay careful attention to the role of the center $Z=F^\times \times F^\times$ of $\GL_2(F)\times B^\times$.   Choose a character $\omega$ of $F^\times=F^\times\times 1$ which extends  $\rho_\FS\vert_{(F^\times\times 1)\cap \LL_{\FS}^\times}$.  We will give a recipe for an extension of $\rho_{\FS}$ to the  group $\mathcal{K}_\FS=\Delta(E^\times)Z\LL_{\FS}^\times\subset \GL_2(F)\times B^\times$ whose restriction to $Z$ is $(g,h)\mapsto \omega(gh^{-1})$.   Call this representation $\rho_{\FS,\omega}$.

Let $\Pi_{\FS,\omega}$ be the compactly supported induction of
$\rho_{\FS,\omega}$ up to $\GL_2(F)\times B^{\times}$.  Then $\Pi_{\FS,\omega}$ is the direct sum of irreducible representations $\pi\otimes\pi'$ of $\GL_2(F)\times B^\times$;  here $\pi$ must have central character $\omega$ and $\pi'$ must have central character $\omega^{-1}$.  We show that $\Pi_{\FS}$
realizes the Jacquet-Langlands correspondence relative to the stratum $\FS$ and the character $\omega$ in
the following sense.  First, we show that a representation $\pi$ of
$\GL_2(F)$ (resp., $B^{\times}$) of central character $\omega$ (resp., $\omega^{-1}$) appears in $\Pi_S$ if and only if $\pi$
(resp., $\check{\pi}$) contains $\FS$ (resp., $\FS'$). Then, we show that an irreducible admissible
representation $\pi\otimes\check{\pi}'$ of $\GL_2(F)\times B^{\times}$
appears inside of $\Pi_{\FS}$ if and only if the epsilon factors of
$\pi$ and $\pi'$ agree up to a minus sign:
\begin{equation}
\label{agreement} \eps(\pi\chi,s,\psi)=-\eps(\pi'\chi,s,\psi).
\end{equation}
Here $\chi$ runs through sufficiently many characters of $F^{\times}$ to
determine $\pi$ from $\pi'$ uniquely. Therefore if $\pi$ is a given
supercuspidal irreducible representation of $\GL_2(F)$ which
contains the stratum $\FS$, then $\Hom_{
\GL_2(F)}\left(\pi,\Pi_\FS\right)$ is a sum of copies of a single
supercuspidal representation $\pi'$ of $B^{\times}$. Then the
contragredient representation of $\pi'$ is the one corresponding to
$\pi$ under the Jacquet-Langlands correspondence.

The linking orders $\LL_{\FS} $ are constructed in
Section~\ref{linkingorderssection}.  We also define corresponding
additive characters $\psi_{\FS}$ of the ring $M_2(F)\times B$ for
which the $\OO_F$-module
$$\LL_{\FS}^*=\set{x\in M_2(F)\times B\biggm{\vert} \psi_{\FS}
(x\LL_\FS)=1}$$ happens to be a two-sided ideal in $\LL_\FS$.  The
required representation $\rho_\FS$ of $\LL_{\FS}^{\times}$ is inflated from
a representation of the unit group of the finite $k$-algebra $\mathcal{R}_\FS=
\LL_{\FS}/\LL_{\FS}^*$. The additive character $\psi_\FS$
descends to a nondegenerate additive character of this ring, so that
we have a theory of Fourier transforms $f\mapsto\CF_\FS f$ for
functions $f$ on $\mathcal{R}_\FS$.  The characteristic property of
$\rho_\FS$ is that its matrix coefficients $f$, considered as
functions on $\mathcal{R}_\FS$ supported on $\mathcal{R}^{\times}_\FS$,  satisfy the
functional equation
\begin{equation}
\label{finitefourier} \CF_\FS f(y)= \pm f(y^{-1})
\end{equation}
 for $y\in\mathcal{R}^{\times}_\FS$;  see Prop.~\ref{replevel0} and Theorem~\ref{rhonu}.  (The sign in this equation depends only on $\FS$.)
The functional equation in Eq.~\ref{finitefourier} on the level of
finite rings is used in Section~\ref{proof} to deduce the functional
equation in Eq.~\ref{agreement} concerning constituents of the
induced representation of $\rho_\FS$ up to $\GL_2(F)\times B^{\times}$.

The reader may be wondering if this sort of strategy may be extended
to the general case of $\GL_n$, where one still lacks a complete
local proof of the existence of the correspondences.  It will not
be difficult to extend the definitions of $\LL_{\FS}$, $\rho_{\FS}$,
and $\Pi_\FS$ to this context.  In doing so one would produce a
recipe for some sort of correspondence $\pi'\mapsto\pi$ for $\pi$
supercuspidal which satisfies Eq.~\ref{agreement} for a certain collection of characters $\chi$.  For $n=3$, we do not know if this collection of characters is enough to characterize the Jacquet-Langlands correspondence.  And for $n>4$, the establishment of Eq.~\ref{agreement} for {\em all} characters is not enough to characterize the correspondence.  One would have to work
harder to obtain access to the characters of the representations
$\pi$ and $\pi'$ so constructed in order to prove the right
identity.

The present effort fits into a larger program concerning the
geometry of Lubin-Tate curves.  Suppose $F$ has uniformizer $\pi_F$ and residue field $k$.  Let $\mathcal{F}_0$ be a formal $\OO_F$-module of height 2 over the algebraic closure of the residue field $k$ of $F$.  For each $m\geq 0$, consider the functor that assigns to each complete local Noetherian $\hat{\OO}_{F^{\text{nr}}}$-algebra $A$ having residue field $\overline{k}$ the set of formal $\OO_F$-modules $\mathcal{F}$ over $A$ equipped with an isomorphism $\mathcal{F}_0\to\mathcal{F}_{\overline{k}}$ and a Drinfeld $\pi_F^m$-level structure.  This functor is represented by a formal curve $X_m$ over $\hat{\OO}_{F^{\text{nr}}}$.   The inverse system of curves $(X_m)_{m\geq 1}$ admits an action by a subgroup $\mathcal{G}$ of the triple product group $\GL_2(F)\times B^\times\times W_F$ of ``index $\Z$".   It is known by the theorems of Deligne and
Carayol, see~\cite{Carayol:ladicreps2}, that the $\ell$-adic \'etale cohomology of
this curve realizes (up to some benign modifications) both the Jacquet-Langlands correspondence
$\pi'\mapsto\pi$ and the local Langlands correspondence
$\sigma\mapsto\pi(\sigma)$ for the discrete series of $\GL_2(F)$.

It would be very interesting to compute a system of semistable models of the curves $X_m$ over a ramified extension of $\hat{\OO}_{F^\text{nr}}$;  then the special fiber of the system ought to realize the
supercuspidal parts of the correspondences in its cohomology.   This has already been done in the ``level 0" case by Bouw-Wewers~\cite{Wewers};  the generalization of the level 0 case for $\GL_n$ was carried out by Yoshida~\cite{yoshida}.   But for higher levels the structure of this special fiber is still unknown.  Ignore the Weil group for the moment and consider the action of $(\GL_2(F)\times B^\times)\cap \mathcal{G}$ on the semi-stable reduction of the system $(X_m)_{m\geq 1}$.  We conjecture that for a simple stratum $\FS$ arising from an elliptic element $\beta\in\GL_2(F)$, the special fiber contains a smooth component $X_\FS$ whose stabilizer is exactly $\Delta(E^\times)\mathcal{L}_\FS^\times$, such that for primes $\ell\neq p$, the $\ell$-adic versions of the representations $\rho_\FS$ appear in the action of this group on $H^1(X_\FS,\QQ_\ell)$.   In light of the preceding paragraphs this would be consistent with the theorems of Deligne-Carayol.  In future work we intend to give a candidate for the structure of the special fiber of the stable reduction of $X_m$ which includes the action of the Weil group $W_F$.

\section{Preparations:  The representation theory of $\GL_2(F)$ and $B^{\times}$}

\subsection{Basic Notations}  In this paper, $F$ will be a finite extension of $\Q_p$, or else a finite extension of $\F_p((T))$.   For a finite extension $E$ of $F$ (possibly $F$ itself), we use the notation $\OO_E$, $\gp_E$, and $k_E$ for the ring of integers, maximal ideal, and quotient field of $E$.  Let $q_E=\# k_E$, and let $q=q_F$.  We fix a uniformizer $\pi_F$ for $F$.   Let $\abs{\;}_F$ be the absolute value on $F^*$ for which $\norm{\pi_F}=q^{-1}$.

We also fix a character $\psi_F$ of $F$ of level 1;  this means that $\psi_F$ vanishes on $\gp_F$ but not on $\OO_F$.

Let $B/F$ be a division algebra of dimension 4; this is unique up to
isomorphism. Let $\OO_B$ be its unique maximal order.  We use
$N_{B/F}$ and $\tr_{B/F}$ to denote the reduced norm and trace,
respectively, from $B$ to $F$; sometimes we will omit the ``$B/F$"
from this notation. If $G$ is the group $\GL_2(F)$ or $B^{\times}$, and
$g\in G$, we will use the notation $\norm{g}_G$ to mean $\abs{\det
g}_F$ or $\abs{\N g}_F$ as appropriate.

Let $A$ be the algebra $M_2(F)$ or $B$.  For any additive character
$\psi$ of $F$, let $\psi_A$ be the character of $A$ defined by
$\psi_A(x)=\psi(\tr_{A/F} x)$.   Let $\mu_{\psi_A}$ (or just
$\mu_\psi$) be the measure on $A$ which is self-dual with respect to
$\psi$.


 Let $\mu_\psi^{\times}$ be the corresponding Haar measure on
$A^{\times}$: $\mu_\psi^{\times}(g)=\norm{\det g}^{-2}_G\mu_\psi(g)$.

\subsection{Chain Orders and Strata}
\label{chainorders}

In this subsection, $A$ is the algebra $M_2(F)$ or $B$.  We will
closely follow the notation of~\cite{Henniart:Bushnell} concerning
chain orders and strata for $\GL_2$, where the situation is somewhat
simpler than the general case of $\GL_n$.

First consider the case $A=M_2(F)$.  A {\em lattice chain} is an
$F$-stable family of lattices $\Lambda=\set{L_i}$ with each
$L_i\subset F\oplus F$ an $\OO_F$-lattice and $L_{i+1}\subset L_i$,
all integers $i$. Let $e(\Lambda)$ be the unique integer for
which $\pi_FL_i=L_{i+e(\Lambda)}$.  Let
$\mathfrak{A}_{\Lambda}$ be the stabilizer in $A$ of
$\Lambda$;  that is, $\mathfrak{A}_{\mathcal{L}}=\set{a\in
A\vert aL_i\subset L_i,\text{ all }i}$.   A {\em chain order} in $A$
is an $\OO_F$-order $\mathfrak{A}\subset A$ equal to
$\mathfrak{A}_{\mathcal{L}}$ for some lattice chain $\Lambda$.
We set $e_{\mathfrak{A}}=e_{\Lambda}$.

For example, suppose $E/F$ is a quadratic field extension of
ramification index $e$.  Identify $E$ with $F\oplus F$ as $F$-vector
spaces. Then $\Lambda=\set{\gp^i_E}$ is a lattice chain with
$e_\mathcal{L}=e$.  Up to conjugation by an element of $A^{\times}$, every
lattice chain arises in this manner.  We have the following
description of $\mathfrak{A}$, again only up to $A^{\times}$-conjugation:
$$\mathfrak{A}=\begin{cases} M_2(\OO_F),&e=1,\\
\tbt{\OO_F}{\OO_F}{\gp_F}{\OO_F},&e=2.
\end{cases}$$
Note also that $\mathfrak{A}^{\times}\subset A^{\times}$ is normalized by
$E^{\times}\subset\GL_2(F)$, and that $\OO_E\subset\mathfrak{A}$.

For a chain order $\mathfrak{A}\subset M_2(F)$, let $\mathcal{K}_{\mathfrak{A}}$ be its normalizer in $\GL_2(F)$.
This equals $F^*M_2(\OO_F)$ if $e_{\mathfrak{A}}=1$.  If $e_{\mathfrak{A}}=2$ then $\mathcal{K}_{\mathfrak{A}}$ is the semidirect product of $\mathfrak{A}^\times$ with the cyclic group generated by a prime element of $\mathfrak{A}$.

Let $\mathfrak{P}_{\mathfrak{A}}$ be the Jacobson radical of $\mathfrak{A}$:  this equals $\pi_FM_2(\OO_F)$ for $\mathfrak{A}=M_2(\OO_F)$ and $\tbt{\gp_F}{\OO_F}{\gp_F}{\gp_F}$ in the case that $\mathfrak{A}=\tbt{\OO_F}{\OO_F}{\gp_F}{\OO_F}$.  We have a filtration of $\mathfrak{A}^\times$ by the subgroups $U_{\mathfrak{A}}^n=1+\mathfrak{P}_{\mathfrak{A}}^n$.  This filtration is normalized by $\mathcal{K}_{\mathfrak{A}}$.
All of the above constructions have obvious (and simpler) analogues in the quaternion algebra $B$:  If $\mathfrak{A}=\OO_B$ is the maximal order in $B$, then the normalizer of $\mathfrak{A}^\times$ in $B^\times$ is all of $B^\times$.  The Jacobson radical $\mathfrak{P}_\mathfrak{A}$ is the unique maximal two-sided ideal of $\mathfrak{A}$, generated by a prime element $\pi_B$; we let $U_{\mathfrak{A}}^n=1+\mathfrak{P}_\mathfrak{A}$ and $e_{\mathfrak{A}}=2$.



\begin{defn}  Let $A$ be the matrix algebra $M_2(F)$ or the quaternion algebra
$B$.  A {\em stratum} in $A$ is a triple $(\mathfrak{A},n,\alpha)$, where $\mathfrak{A}$
is a chain order if $A=M_2(F)$ (resp. $\OO_B$ if $A=B$), $n$ is an
integer, and $\alpha\in \gP_{\mathfrak{A}}^{-n}$.  Two strata
$(\mathfrak{A},n,\alpha)$ and $(\mathfrak{A},n,\alpha')$ are
equivalent if $\alpha\equiv\alpha'\pmod{\gP^{1-n}}$.    The stratum
$(\mathfrak{A},n,\alpha)$ is {\em ramified simple} if $E=F(\alpha)$
is a ramified quadratic extension of $F$, $n$ is odd, and $\alpha\in
E$ has valuation exactly $-n$.   The stratum is {\em unramified
simple} if $E$ is an unramified quadratic extension of $F$,
$\alpha\in E$ has valuation exactly $-n$, and the minimal polynomial
of $\pi_F^n\alpha$ is irreducible mod $\gp_F$.  Finally, the stratum
is {\em simple} if it is ramified simple or unramified simple.
\end{defn}

There is a correspondence $\mathfrak{S}'\mapsto\mathfrak{S}$ between
simple strata in $B$ and simple strata in $M_2(F)$.  Given the
simple stratum $\mathfrak{S}'=(\mathfrak{A}',n',\alpha')$, let
$E=F(\alpha')$. Choose an embedding $E\injects M_2(F)$, and let
$\alpha$ be the image of $\alpha'$.  Finally, let
$\mathfrak{A}\subset M_2(F)$ be a chain order associated to $E$.
Then $\mathfrak{S}=(\mathfrak{A},n,\alpha)$.  The correspondence
$\mathfrak{S}'\to\mathfrak{S}$ is a bijection between conjugacy
classes of simple strata in $B$ and in $M_2(F)$, respectively.  The
relationship between $n'$ and $n$ is as follows:  $n'=n$ if $E/F$ is
ramified and $n'=2n$ if $E/F$ is unramified.

Let $\pi$ be an irreducible admissible representation of $\GL_2(F)$.
The level $\ell(\pi)$ is defined to be the least value of $n/e$,
where $(n,e)$ runs over pairs of integers for which there exists a
chain order $\mathfrak{A}$ of ramification index $e$ such that $\pi$
contains the trivial character of $U_{\mathfrak{A}}^{n+1}$.  If
$\pi$ is a representation of $B^{\times}$, we define $\ell(\pi)$ to be
$n/2$, where $n$ is the least integer for which $\pi$ contains the
trivial character of $U_{\OO_B}^{n+1}$.

We shall call $\pi$ {\em minimal} if its level cannot be lowered by
twisting by one-dimensional characters of $F^{\times}$.

When $n\geq 1$, a stratum $\FS=(\mathfrak{A},n,\alpha)$ of $M_2(F)$
or $B$ determines a nontrivial character $\psi_\alpha$ of
$U_{\mathfrak{A}}^n/U_{\mathfrak{A}}^{n+1}$ by
$\psi_\alpha(1+x)=\psi_F(\tr_{A/F}(\alpha x))$.  This character only
depends on the equivalence class of $\FS$.

If $\FS$ is a stratum, we say that $\pi$ {\em contains the stratum}
$\FS$ if $\pi\vert_{U_{\mathfrak{A}}^n}$ contains the character
$\psi_\alpha$.  From~\cite{Henniart:Bushnell}, 14.5 Theorem, we have
the following classification of supercuspidal representations of
$\GL_2(F)$:

\begin{Theorem} \label{classification1} A minimal irreducible representation $\pi$ of $\GL_2(F)$ is supercuspidal if and only if exactly one of the following conditions holds:
\begin{enumerate}
\item $\pi$ has level 0, and $\pi$ is contains a representation of $\GL_2(\OO_F)$ inflated from an irreducible cuspidal representations of $\GL_2(k_F)$.
\item $\pi$ has level $\ell>0$, and $\pi$ contains a simple stratum.
\end{enumerate}
\end{Theorem}

The classification of representations of $B^{\times}$ is analogous:

\begin{Theorem} \label{classificationB} A minimal irreducible
representation $\pi$ of $B^{\times}$ of dimension greater than one
satisfies exactly one of the following properties:
\begin{enumerate}
\item $\pi$ has level 0, and $\pi$ contains a representation of
$\OO_B^{\times}$ inflated from a character $\chi$ of $k_B^{\times}$ not factoring
through the norm map $k_B^{\times}\to k^{\times}$.
\item $\pi$ has level $\ell>0$, and $\pi$ contains a simple stratum.
\end{enumerate}
\end{Theorem}
By $k_B$ we mean the finite field $\OO_B/\gP_B$:  this is a
quadratic extension of $k$.

The supercuspidal representations of $\GL_2(F)$ and $B^{\times}$ are all
induced from irreducible representations of open compact-mod-center
subgroups in a manner which can be made explicit.  Suppose
$\FS=(\mathfrak{A},n,\alpha)$ is a simple stratum in $M_2(F)$ or
$B$.    Let $E\subset \GL_2(F)$ be the subfield $F(\alpha)$. The
definition of $\psi_\alpha$ given above is well-defined on the
subgroup $U_{\mathfrak{A}}^{\floor{n/2}+1}$. Let $J_{\FS}\subset
\GL_2(F)$ denote the group $E^{\times}U_{\mathfrak{A}}^{\floor{(n+1)/2}}$
and let $C(\psi_\alpha, \mathfrak{A})$ denote the set of isomorphism
classes of irreducible representations $\Lambda\in\hat{J}_\FS$
for which $\Lambda\vert_{U_{\mathfrak{A}}^{\floor{n/2}+1}}$ is a
multiple of $\psi_\alpha$.

\begin{defn} A {\em cuspidal inducing datum} in $A^{\times}$
is a pair $(\mathfrak{A},\Xi)$, where $\mathfrak{A}$ is a chain
order in $A$ and $\Xi$ is a representation of
$\mathcal{K}_{\mathfrak{A}}$ of one of the following types:
\begin{enumerate}
\item $A=M_2(F)$, $\mathfrak{A}\isom M_2(\OO_F)$, and the restriction of
$\Xi$ to $\GL_2(\OO_F)$ is inflated from a cuspidal representation
of $\GL_2(k)$.
\item $A=B$, and the restriction of $\Xi$ to $\OO_B^{\times}$ contains a
character of inflated from a character of $k_B^{\times}$ not factoring
through the norm map $k_B^{\times}\to k^{\times}$.
\item There is a simple stratum $(\mathfrak{A},n,\alpha)$ and a
representation $\Lambda\in C(\psi_\alpha,\mathfrak{A})$ for which
$\Xi=\Ind_{J_{\FS}}^{\mathcal{K}_{\mathfrak{A}}}\Lambda$.
\end{enumerate}
In the first two cases we will say that $(\mathfrak{A},\Xi)$ has
level zero.  In the third case, we will say that
$(\mathfrak{A},\Xi)$ has level $n$.
\end{defn}

The following construction of supercuspidal representations is found
in Section 15.5 of~\cite{Henniart:Bushnell} in the case of
$A=M_2(F)$:

\begin{Theorem} \label{classification2} If $(\mathfrak{A},\Xi)$ is a
cuspidal inducing datum then
$\pi_\Xi=\Ind_{\mathcal{K}_{\mathfrak{A}}}^{A^{\times}}\Xi$ is an
irreducible minimal supercuspidal representation of $A^{\times}$.
Conversely, every minimal supercuspidal representation of $A^{\times}$
arises in this manner. The cuspidal inducing datum
$(\mathfrak{A},\Xi)$ has level zero if and only if $\pi$ has level
zero.  Furthermore, $(\mathfrak{A},\Xi)$ arises from the simple
stratum $\FS$ if and only if $\pi$ contains $\FS$.
\end{Theorem}


\subsection{Zeta functions and local constants}

In this section we follow Godement and Jacquet
~\cite{Godement:Jacquet}. Let $A$ be the algebra $B$ or $M_2(F)$,
and let $G=A^{\times}$. Let $\psi\in\hat{F}$ be an additive character of
$F$. Let $\pi$ be a supercuspidal (not necessarily irreducible)
representation of $G$, realized on the space $W$. When $w\in W$,
$\check{w}\in \check{W}$, we let $\gamma_{\check{w},w}\from G\to\C$
denote the function
$$g\mapsto \class{\check{w},\pi(g)w}.$$  Let $\mathcal{C}(\pi)$
denote the $\C$-span of the functions $\gamma_{\check{w},w}$ for
$w\in W$, $\check{w}\in \check{W}$. These functions are compactly
supported modulo the center $Z$ of $G$.

Let $C^\infty_c(A)$ be the space of locally constant compactly supported complex-valued functions on $A$.  For $\Phi\in C^\infty_c(A)$ and $f\in \mathcal{C}(\pi)$, define
the zeta function
$$\zeta(\Phi,f,s)=\int_G \Phi(g)f(g)\norm{g}^s\;d\mu^{\times}_\psi(g).$$
When $\pi$ is irreducible (and still cuspidal), there is a rational
function $\eps(\pi,s,\psi)\in\C(q^{-s})$ satisfying
$$\zeta\left(\hat{\Phi},\check{f},\tfrac{3}{2}-s\right)=\eps(\pi,s,\psi)\zeta\left(\Phi,f,\tfrac{1}{2}+s\right),$$ where $\hat{\Phi}$ is the Fourier transform of $\Phi$ with respect to $\psi$.  (Since $\pi$ is cuspidal, its $L$-function vanishes.)

The local constant further satisfies
\begin{equation}
\label{duality}
\eps(\pi,s,\psi)\eps(\check{\pi},1-s,\psi)=\omega_\pi(-1)
\end{equation}
where $\omega_\pi$ is the central character of $\pi$.

\subsection{Converse Theory}


By the converse theorem, a supercuspidal representation of
$\GL_2(F)$ of $B^\times$ is determined by the epsilon factors of all of its twists by
one-dimensional characters. We need an effective version of this
theorem, which states that a supercuspidal representation is determined up to isomorphism by the data of its level together with the epsilon factors of
twists of $\pi$ by a collection of characters of $F^{\times}$ of bounded
level.


Next, we observe that epsilon factors have the ``stability" property.  If $\chi$ is a character of $F^*$, let the level $\ell(\chi)$ be the least integer $n$ such that $\chi$ vanishes on $1+\gp_F^{n+1}$.  Then if $\pi$ is an irreducible representation of $\GL_2(F)$ or $B^\times$, and $\chi$ is a character of $F^\times$ with $\ell(\chi)>\ell(\pi)$, then $\eps(\pi\chi,s,\psi)$ only depends on $\chi$ and the central character of $\pi$ (and of course $\psi$).  This is Prop. 3.8 of~\cite{JacquetLanglands} in the case of $\GL_2(F)$ and Prop. 2.2.5 of~\cite{GerardinLi} in the case of $B^\times$.

As $\chi$ varies through all characters of $F^\times$, the quantities $\eps(\chi\pi,s,\psi)$ determine $\pi$ up to isomorphism.  We may therefore conclude the following explicit converse theorem:
\begin{Theorem}\label{converse}  Let $\pi_1$ and $\pi_2$ be two minimal
supercuspidal representations of $\GL_2(F)$ or $B^{\times}$ having
the same central character and equal level $\ell$.
Then
$\pi_1\isom\pi_2$ if and only if
\begin{equation}
\label{epseq}
 \eps(\pi_1\chi,s,\psi)=\eps(\pi_2\chi,s,\psi)
\end{equation}
for all characters $\chi\in\hat{F}^{\times}$ for which $\ell(\chi)\leq \ell$.
\end{Theorem}

\begin{defn}
For minimal supercuspidal representations $\pi'$ and $\pi$ of $B^{\times}$ and $\GL_2(F)$ having
the same central character, we say that $\pi'$ and $\pi$ {\em
correspond} if the following conditions hold:
\begin{enumerate}
\item $\pi$ and $\pi'$ have the same level $\ell$.
\item The equation
$$\eps(\pi\chi,s,\psi)=-\eps(\pi'\chi,s,\psi)$$ holds for all characters $\chi$ with $\ell(\chi)\leq \ell$.
\end{enumerate}
\end{defn}

In view of Theorem~\ref{converse}, at most one $\pi$ can correspond
to a given $\pi'$, and vice versa.

\section{Zeta functions for $GL_2(F)\times B^{\times}$.}

In this section we adopt the abbreviations $A_1=M_2(F)$, $A_2=B$,
$G_1=GL_2(F)$, $G_2=B^{\times}$.

Let $\mathbf{A}=A_1\times A_2$.  Let
$\mathbf{G}=\mathbf{A}^{\times}=\GL_2(F)\times B^{\times}$.  We will define zeta
functions for representations of $\mathbf{G}$ and use them to give a
criterion for when such a representation ``realizes the
Jacquet-Langlands correspondence."  We will adopt the convention
that if $g\in \mathbf{G}$, then $g_1$ and $g_2$ are its projections
in $\GL_2(F)$ and $B^{\times}$ respectively.  Let $\Pi$ be an admissible
cuspidal representation of $\mathbf{G}$.  For $\Phi\in
C_c^\infty(\mathbf{A})$ and $f\in\mathcal{C}(\Pi)$, define the zeta
function
$$\zeta(\Phi,f,s)=\int_{\mathbf{G}} \Phi(g)f(g)\norm{g_1}^{s}\norm{g_2}^{2-s}d\mu^{\times}(g),$$ where $\mu^{\times}$ is a Haar measure on $\mathbf{G}$.

Let $\psi$ be an additive character of $F$, and let
$\mu^{\times}_\psi=\mu^{\times}_{\mathbf{A},\psi}=\mu_{A_1,\psi}^{\times}\times\mu_{A_2,\psi}^{\times}$;
this is a Haar measure on $\mathbf{G}$.  Let $\psi_\mathbf{A}$ be
the additive character
$(x_1,x_2)\mapsto\psi_{A_1}(x_1)\psi_{A_2}(-y_1)$.  The Fourier transform of a decomposable test function $\Phi=\Phi_1\otimes\Phi_2\in C_c^\infty(\mathbf{A})$ is $\hat{\Phi}(x_1,x_2)=\hat{\Phi}_1(x_1)\hat{\Phi}_2(-x_2).$
Consequently if $f=f_1\otimes f_2\in \mathcal{C}(\pi_1\otimes\pi_2)$
is a decomposable matrix coefficient for a tensor product
representation $\pi_1\otimes\pi_2$ of $\mathbf{G}$, then
\begin{equation}
\label{doublezeta}
\zeta(\hat{\Phi},f,s)=\omega_{\pi_2}(-1)\zeta(\hat{\Phi}_1,f_1,s)\zeta(\hat{\Phi}_2,f_2,2-s),
\end{equation}
where $\omega_{\pi_2}$ is the central character of $\pi_2$.

\begin{prop} \label{FE} Let $\Pi$ be an admissible cuspidal semisimple (not
necessarily irreducible) representation of $\GL_2(F)\times B^{\times}$.
The following are equivalent:
\begin{enumerate}
\item For every irreducible representation $\pi_1\otimes\pi_2$ of
$\GL_2(F)\times B^{\times}$ appearing in $\Pi$, we have
$$\eps(\pi_1,s,\psi)=-\eps(\check{\pi}_2,s,\psi).$$
\item The functional equation
\begin{equation}
\label{BIGFE} \zeta(\Phi,f,s)=-\zeta(\hat{\Phi},\check{f},2-s)
\end{equation}
holds for all $\Phi\in C_c^{\infty}(\mathbf{A})$,
$f\in\mathcal{C}(\Pi)$.  (Here the integral is taken with respect to
the measure $\mu^{\times}_{\mathbf{A},\psi}$, and the Fourier transform is
taken with respect to the character $\psi_{\mathbf{A}}$.)
\end{enumerate}
\end{prop}


\begin{proof} It will simplify our notation if we set $s_1=s$,
$s_2=2-s$.  Let $\pi_1\otimes\pi_2$ be any irreducible
representation of $G_1\times G_2$ appearing in $\Pi$. For $i=1,2$,
let $\Phi_i\in C^\infty_c(G_i)$ and $f_i\in\mathcal{C}(\pi_i)$ be
such that $\zeta(\Phi_i,f_i,s_i)\neq 0$.  Let
$\Phi=\Phi_1\otimes\Phi_2$ and $f=f_1\otimes f_2$.  The respective
functional equations for $\pi_1$ and $\pi_2$ are
$$\zeta(\hat{\Phi}_i,\check{f}_i,2-s_i)=\eps\left(\pi_i,s_i-\tfrac{1}{2},\psi\right)\zeta(\Phi_i,f_i,s_i),\;i=1,2.$$
Multiplying these together and applying Eq.~\ref{doublezeta} yields
$$\omega_{\pi_2}(-1)\zeta(\hat{\Phi},
\check{f},2-s)=\eps\left(\pi_1,s-\tfrac{1}{2},\psi\right)\eps\left(\pi_2,\tfrac{3}{2}-s,\psi\right)\zeta(\Phi,
f,s).$$ Therefore Eq.~\ref{BIGFE} holds if and only if
$$\eps\left(\pi_1,s-\tfrac{1}{2},\psi\right)\eps\left(\pi_2,\tfrac{3}{2}-s,\psi\right)=-\omega_{\pi_2}(-1).$$
Combining this with the standard relation
$$\eps\left(\pi_2,\tfrac{3}{2}-s,\psi\right)\eps\left(\check{\pi}_2,s-\tfrac{1}{2},\psi\right)=\omega_{\pi_2}(-1)$$
yields
$$\eps\left(\pi_1,s-\tfrac{1}{2},\psi\right)=-\eps\left(\check{\pi}_2,s-\small{\tfrac{1}{2}},\psi\right).$$
We see now that (2)$\implies$(1): Apply Eq.~\ref{BIGFE} to an
arbitrary matrix coefficient $f=f_1\otimes f_2$ belonging to
$\pi_1\otimes\pi_2\subset \Pi$.  For the converse, one need only
note that every $\Phi\in C_c^{\infty}(\mathbf{A})$ and $f\in
\mathcal{C}(\Pi)$ is a finite sum of pure tensors, and
$\zeta(\Phi,f,s)$ is linear in $\Phi$ and $f$.
\end{proof}

Combining Prop.~\ref{FE} with the Converse Theorem~\ref{converse}
gives a necessary and sufficient condition for a representation
$\Pi$ of $\GL_2(F)\times B^{\times}$ to realize the Jacquet-Langlands
correspondence.  When $f\in C_c^\infty(\mathbf{G})$ and
$\chi\in\hat{F}^{\times}$, we let $\chi f$ be the function
$g\mapsto\chi(\det(g_1)\N(g_2)^{-1})f(g)$.
\begin{Cor} \label{FEcor} Let $\Pi$ be an admissible cuspidal semisimple representation of $\GL_2(F)\times B^{\times}$ on which the
diagonally-embedded group $\Delta(F^{\times})$ acts trivially. Assume
either that every irreducible representation of $\GL_2(F)$ (resp.,
$B^{\times}$) appearing in $\Pi$ is minimal of the same level $\ell$. Then the
following are equivalent:
\begin{enumerate}
\item $\Pi$ is the direct sum of irreducible representations of $\mathbf{G}$ of the form
$\pi_1\otimes\check{\pi}_2$, where $\pi_1$ and $\pi_2$ correspond.
\item The functional equation
\begin{equation}
\label{BIGFEchi} \zeta(\Phi,\chi
f,s)=-\zeta(\hat{\Phi},\chi^{-1}\check{f},2-s) \end{equation} holds
for all $\Phi\in C_c^{\infty}(\mathbf{A})$, $f\in\mathcal{C}(\Pi)$,
and for all characters $\chi\in\hat{F}^{\times}$ for which $\ell(\chi)\leq \ell$.
\end{enumerate}
\end{Cor}
\begin{proof}  That (1)$\implies$(2) is clear from Prop.~\ref{FE}.
Therefore assume (2).  Suppose $\pi_1\otimes\check{\pi}_2$ appears
in $\Pi$.  Since $\Pi$ vanishes on $\Delta(F^{\times})$, the central
characters of $\pi_1$ and $\pi_2$ agree.  By Prop.~\ref{FE} we find
that $\eps(\pi_1\chi,s,\psi)=-\eps(\pi_2\chi,s,\psi)$ for all
characters $\chi$ of level no greater than $\ell$, so $\pi_1$ and $\pi_2$
correspond.
\end{proof}

\section{Linking orders and congruence subgroups of $\GL_2(F)\times
B^{\times}$ } \label{linkingorderssection}

Our goal now is to produce, for each simple stratum $\FS$ in
$M_2(F)$, a certain semisimple representation $\Pi_\FS$ of $\GL_2(F)\times B^{\times}$
having the following properties:

\begin{enumerate}
\item $\Pi_\FS$ vanishes on the diagonal subgroup
$\Delta(F^{\times})\subset\GL_2(F)\times B^{\times}$.
\item The restriction of $\Pi_\FS$ to the first factor $\GL_2(F)$
is a sum of exactly those irreducible representations which contain
$\FS$.  Similarly, the restriction of $\Pi_\FS$ to the second factor
$B^{\times}$ is a sum of exactly those irreducible representations of $B^{\times}$
which contain the corresponding stratum $\FS'$ in $B$.
\item Matrix coefficients for $\Pi_\FS$ satisfy the functional
equation in Eq.~\ref{BIGFEchi} for sufficiently many $\chi$.
\end{enumerate}
We will present a similar construction for representations of level
zero.  In light of Cor.~\ref{FEcor}, such a family $\set{\Pi_\FS}$
is sufficient to establish the Jacquet-Langlands correspondence.

The strategy for producing $\Pi_\FS$ is as follows:  We will first
define an certain order $\LL_\FS\subset M_2(F)\times B$. The
required representation $\Pi_\FS$ will be induced from a certain
representation of $\LL_\FS^{\times}$.  In this section we construct
the orders $\LL_\FS$ and gather some geometric properties in
preparation for proving the properties listed above.

\subsection{Geometric preparations: $M_2(F)$ and $B$}
\label{geomprep}

Let $E/F$ be a separable quadratic extension field of ramification
degree $e$.  Let $\OO_E$ be its ring of integers, $\gp_E$ its
maximal ideal, $k_E$ its quotient field and $\sigma$ the nontrivial
element of $\Gal(E/F)$.

Let $A$ be the ring $M_2(F)$ or $B$.  Define an order
$\mathfrak{A}\subset A$ as follows:  if $A=M_2(F)$, let
$\mathfrak{A}$ be the chain order equal to the endomorphism ring of
the lattice chain $\set{\gp_E^i}$, as in Section~\ref{chainorders}.
If $A=B$, let $\mathfrak{A}=\OO_B$. Either way, we may identify
$\OO_E$ with an $\OO_F$-subalgebra of $\mathfrak{A}$ in such a way
that $\mathfrak{A}\cap E = \OO_E$.



There is a nondegenerate pairing $A\times A\to F$ given by
$(x,y)\mapsto \tr_{A/F}(xy)$.   Let $C$ be the complement of
$E$ in $A$ with respect to this pairing, so that $A=E\oplus
C$.  Let $s_A\from A\to E$ be the projection onto the first
factor.   Note that both the space $C$ and the map $s_A$ are
stable under multiplication by $E$ on either side. $C$ is a
(left and right!) $E$-vector space of dimension 1. It satisfies the
property that $\alpha v = v\alpha^\sigma$ for all $v\in C$,
$\alpha\in E$.  Let $\mathfrak{C}=\mathfrak{A}\cap C$.
\begin{lemma}\label{AA}  We have
$$\mathfrak{C}\mathfrak{C}=\begin{cases} \gp_E,&\text{$E/F$ unramified and $A=B$}\\
\OO_E,&\text{ all other cases.}\end{cases}$$
\end{lemma}
\begin{proof}
Since elements of $E$ commute with
$\mathfrak{C}\mathfrak{C}$, we must have
$\mathfrak{C}\mathfrak{C}\subset E$;  since
$\mathfrak{C}\subset\mathfrak{A}$ this implies
$\mathfrak{C}\mathfrak{C}\subset
E\cap\mathfrak{A}=\OO_E$.  Thus
$\mathfrak{C}\mathfrak{C}$ is an $\OO_E$-submodule of
$\OO_E$;  {\em i.e.} it is an ideal of $\OO_E$.

If $A=M_2(F)$ then $\mathfrak{A}$ is the endomorphism ring of the
lattice chain $\set{\gp_E^i}$.  Consider the element
$\sigma\in\Gal(E/F)$:  this certainly preserves each $\gp_E^i$ and
therefore belongs to $\mathfrak{A}$.  For any $\alpha\in E$, we have
that $(\alpha\sigma)^2=\N_{E/F}(\alpha)$ belongs to the center
$F\subset M_2(F)$, but $\alpha\sigma$ does not itself belong to $F$,
implying that $\tr_{A/F}(\alpha\sigma)=0$ and therefore that
$\sigma\in C$.  So $\sigma\in C\cap
\mathfrak{A}=\mathfrak{C}$.  Consequently
$\mathfrak{C}\mathfrak{C}$ contains $\sigma^2=1$, whence
it is the unit ideal.

Now suppose $A=B$.  Let $v_B\from B^{\times}\to \Z$ denote the valuation on
$B$.  If $E/F$ is ramified, then a uniformizer $\pi_E$ of $E$ has
$v_B(\pi_E)=1$, so that if $x\in\mathfrak{C}$ has valuation
$n$, then $\pi_E^{-n}x\in\mathfrak{C}$ is a unit.  This
implies that $\mathfrak{C}\mathfrak{C}$ is the unit
ideal.

On the other hand if $E/F$ is unramified, then every element of $E$ has even
valuation in $B$.  Considering that $A=E\oplus C$, this means
that $\mathfrak{C}$ contains an element $\pi_B$ of valuation
1, so that $\mathfrak{C}=\OO_E\pi_B$.  Then
$\mathfrak{C}\mathfrak{C}=\OO_E\pi_B^2=\gp_E$ as
required.
\end{proof}

Now suppose that $\FS=(\mathfrak{A},n,\alpha)$ is a simple stratum
in $A$ with $E=F(\alpha)$.  Choose an additive character $\nu$ of
$E$ vanishing on $\gp_E^{n+1}$ but not on $\gp_E^n$. Assume that
$\nu=\nu^\sigma$ if $e=1$.  Then define a character $\nu_\FS$ of $A$
by $\nu_\FS(x)=\nu(s_A(x))$.

Whenever $W$ is an $\OO_E$-stable subspace of $A$, we may define the
annihilator of $W$ with respect to $\nu_\FS$:  $$W^*=\set{x\in
A\;\vert\;\nu_\FS(xW)=1};$$ then $W^*$ is also an $\OO_E$-module.
Note that $(\gp_E^kW)^*=\gp_E^{-k}W^*$.

\begin{lemma}
\label{Aperp} The $\OO_E$-module
$\mathfrak{C}^*$ equals $E\oplus
\gp_E^n\mathfrak{C}$ if $E/F$ is unramified and $A=B$.  It equals  $E\oplus
\gp_E^{n+1}\mathfrak{C}$ in all other cases.
\end{lemma}
\begin{proof}   Certainly we have
$E\subset\mathfrak{C}^*$;  all that remains
is to find $\mathfrak{C}^*\cap
\mathfrak{C}$.  This last is an $\OO_E$-submodule of the free
rank-one $\OO_E$-module $\mathfrak{C}$, so that it equals
$I\mathfrak{C}$ for an ideal $I\subset\OO_E$.  For an element
$x\in \OO_E$ to belong to $I$ the condition is
$\nu_\FS(s_A(x\mathfrak{C}\mathfrak{C}))=\nu(I\mathfrak{C}\mathfrak{C})=1$.
The lemma now follows from Lemma~\ref{AA} and the definition of
$\nu$.
\end{proof}

For an integer $m\geq 1$, we define an $\OO_E$-submodule $V_A^m\subset\mathfrak{C}$ as
follows:
$$V_A^m=\begin{cases}
\gp_E^{\floor{m/2}}\mathfrak{C},&\text{ $A=B$ and $E/F$ unramified}\\
\gp_E^{\floor{(m+1)/2}}\mathfrak{C},&\text{ all other
cases.}\end{cases}$$ The next proposition shows that
$V_A^n\subset\mathfrak{C}$ is nearly a ``square root" of the
ideal $\gp_E^n$:


\begin{prop}  \label{vprop}  The module $V_A^n$ has the following properties:
\begin{enumerate}
\item $V_A^nV_A^n\subset \gp_E^n.$  More precisely, if $E/F$ is unramified then the
value of $V_A^nV_A^n$ is given by the following table:
\begin{center}
\begin{tabular}{c|cc}
& $n$ even & $n$ odd\\
\hline $A=M_2(F)$ & $\gp_E^n$ & $\gp_E^{n+1}$ \\
$A=B$ & $\gp_E^{n+1}$ & $\gp_E^n$
\end{tabular}
\end{center}
\item If $E/F$ is ramified, then $V_A^n=V_A^{n+1}$.
\item If $E/F$ is unramified, then the dimension of $V_A^n/V_A^{n+1}$ as a $k_E$-vector space is given by the following table:
\begin{center}
\begin{tabular}{c|cc}
 & $n$ even & $n$ odd\\
 \hline
 $A=M_2(F)$ & $1$ & $0$ \\
 $A=B$ & $0$ & $1$
\end{tabular}
\end{center}
\item With respect to the character $\nu_\FS$, we have $(V_A^n)^* = E\oplus  V_A^{n+1}$.
\end{enumerate}

\end{prop}

\begin{proof}  Claim (1) follows from Lemma~\ref{AA}.
For claim (2):  Since $E/F$ is ramified, $n$ must be odd by definition of simple
stratum; then $\floor{(n+1)/2}=\floor{((n+1)+1)/2}$. For claim (3),
assume $E/F$ is unramified. When $A=M_2(F)$ we have
$V_A^n=\gp_E^{\floor{(n+1)/2}}\mathfrak{C}$, so that there is
an isomorphism of  $k_E$-vector spaces $V_A^n/V_A^{n+1}\approx
\gp_E^{\floor{(n+1)/2}}/\gp_E^{\floor{(n+2)/2}}$, and this has
dimension 1 or 0 as $n$ is even or odd, respectively.   When
$A=M_2(F)$ we have $V_A^n = \gp_E^{\floor{n/2}}\mathfrak{C}$,
so that there is an isomorphism of $k_E$-vector spaces
$V_2^n/V_2^{n+1}=\gp_E^{\floor{n/2}}/\gp_E^{\floor{(n+1)/2}}$, and
this has dimension 0 or 1 as $n$ is even or odd, respectively.

Claim (4) follows directly from Lemma~\ref{Aperp}.
\end{proof}

\subsection{Congruence subgroups and cuspidal
representations}

Keeping the notations from the previous subsection, we let
\begin{eqnarray*}
H_\FS&=&1+\gp_E^n+V_A^n\\
H_\FS^1&=&1+\gp_E^n+V_A^{n+1}.
\end{eqnarray*}
These are subgroups of $\mathfrak{A}^{\times}$
because $V_A^n$ is an $\OO_E$-module and because $V_1^nV_1^n\subset
\gp_E^n$ by Prop.~\ref{vprop}.  Note the inclusions
$U^n_{\mathfrak{A}}\subset H_\FS^1\subset H_\FS\subset J_\FS$ and $H^1_\FS\subset U^{\floor{n/2}+1}_{\mathfrak{A}}$.

\begin{prop}  \label{HS}
For a representation $\Lambda\in C(\psi_\alpha,\mathfrak{A})$, we
have that $\Lambda\vert_{H_\FS}$ is irreducible.  Further, $\Lambda\vert_{H_\FS}$ is the unique irreducible representation of $H_\FS$ whose restriction to $H_\FS^1$ is a sum of copies of $\psi_\alpha\vert_{H_\FS^1}$.  
\end{prop}

\begin{proof}  If $E/F$ is ramified, the claims in the proposition are trivial, because $H_\FS=H^1_\FS$ and $\Lambda$ is a one-dimensional character.  If $E/F$ is unramified, then the same is true in the case that $A=M_2(F)$ and $n$ is odd, and as well in the case that $A=B$ and $n$ is even.

Therefore assume that $E/F$ is unramified, and that $A=M_2(F)$ and $n$ is even, or else that $A=B$ and and $n$ is odd.  Then $V_A^n/V_A^{n+1}$ is a $k_E$-module of dimension 1.  Let $\psi_\alpha^1$ denote the restriction of $\psi_\alpha$ to $H_\FS^1$.  We
have an exact sequence
$$1\to H_\FS^1/\ker\psi_\alpha^1\to H_\FS/\ker\psi_\alpha^1\to V^n_A/V_A^{n+1}\to 1$$
in which $H_\FS^1/\ker\psi_\alpha^1$ is the center.  Thus $H_\FS/\ker\psi_\alpha^1$ is
a discrete Heisenberg group.  By the discrete Stone-von Neumann
Theorem, there is a unique irreducible representation
$\tilde{\psi}_\alpha$ of $H_\FS$ lying over $\psi_\alpha^1$.

If $\Lambda\in C(\psi_\alpha,\mathfrak{A})$, then
$\Lambda\vert_{H_\FS}$ is a $q$-dimensional representation of
$H_\FS$ whose restriction to $H_\FS^1$ is a multiple of
$\psi_\alpha^1$.  By the uniqueness property of
$\tilde{\psi}_\alpha$, we must have
$\Lambda\vert_{H_\FS}=\tilde{\psi}_\alpha$.  The proposition
follows.
\end{proof}

\subsection{Linking Orders}  It is time to investigate the geometry
of the product algebra $M_2(F)\times B$.  It will be helpful to use
the abbreviations $A_1=M_2(F)$, $A_2=B$, $\mathbf{A}=M_2(F)\times
B$. Suppose $\FS=\FS_1=(\mathfrak{A}_1,n_1,\alpha_1)$ is a simple
stratum in $M_2(F)$. Choose an embedding $E=F(\alpha_1)\injects B$
and let $\alpha_2\in B^{\times}$ be the image of $\alpha_1$ so that
$\FS_2=(\mathfrak{A}_2,n_2,\alpha_2)$ is the simple stratum in $B$
which corresponds to $\FS$. Here $\mathfrak{A}_2=\OO_B$.  For
convenience of notation we set $n=n_1$. Let
$\mathfrak{A}=\mathfrak{A}_1\times\mathfrak{A}_2$ and let
$\Delta\from E\to \mathbf{A}$ be the diagonal map $\Delta(a)=(a,a)$.
We denote by $s_1$ and $s_2$ the projections $A_1\to E$,
$A_2\to E$, respectively.  Let $C_i$ be the complement of $E$ in $A_i$.

Let $\nu$ be an additive character of $E$ as in
Section~\ref{geomprep}. We define a character $\nu_{\FS}$ of
$\mathbf{A}$ by $$\nu_{\FS}(x_1,x_2)=\nu(s_1(x_1)-s_2(x_2)).$$

\begin{lemma}\label{Operp}
With respect to $\nu_{\FS}$, the annihilator of the diagonally
embedded subring $\Delta(\OO_E)\subset\mathfrak{A}$ is
$$\left(\Delta(\OO_E)\right)^* = \Delta(E)+\gp_E^{n+1}\times\gp_E^{n+1}+C_1\times
C_2.$$
\end{lemma}
\begin{proof} Suppose $(x_1,x_2)\in(\Delta(\OO_E))^*$;  then for
all $\beta\in \OO_E$, $v(\beta(s_1(x_1)-s_2(x_2)))=1$.
 This means exactly that $s(x_1)\equiv s(x_2)\pmod{\gp_E^{n+1}}$, so that the pair
$(s(x_1),s(x_2))$,  being equal to
$(s(x_1),s(x_1))+(0,s(x_2)-s(x_1))$, lies in
$\Delta(E)+\gp_E^{n+1}\times\gp_E^{n+1}$ as required.
\end{proof}

Let $\mathbf{V}^n =V_1^n\times V_2^n\subset\mathfrak{A}$.   The
following properties of $\mathbf{V}^n$ follow directly from
Prop.~\ref{vprop}:
\begin{prop} \label{vvprop} The module $\mathbf{V}^n$ has the
following properties:
\begin{enumerate}
\item $\mathbf{V}^n\mathbf{V}^n\subset \gp_E^n\times\gp_E^n.$
Furthermore, if $E/F$ is unramified then $\mathbf{V}^n\mathbf{V}^n$ equals
$\gp_E^n\times\gp_E^{n+1}$ or $\gp_E^{n+1}\times\gp_E^n$ as $n$ is
even or odd, respectively.
\item If $E/F$ is unramified, then $\mathbf{V}^n/\mathbf{V}^{n+1}$ is a left and right $k_E$-vector space of dimension 1, with the property that $\alpha v=v\alpha^q$ for $\alpha\in k_E$, $v\in\mathbf{V}^n/\mathbf{V}^{n+1}$.
\item If $E/F$ is ramified, then $\mathbf{V}^n=\mathbf{V}^{n+1}$.
\item With respect to $\psi_{\FS}$, the annihilator of $\mathbf{V}^n$ is $(E\times E)\oplus \mathbf{V}^{n+1}$.
\end{enumerate}
\end{prop}

\begin{defn}  The {\em linking order} $\LL_{\FS}$ is defined by $$\LL_{\FS}=\Delta(\OO_E)+\gp_E^n\times\gp_E^n+\mathbf{V}^n.$$
\end{defn}



Then $\LL_{\FS}$ is a (left and right) $\OO_E$-submodule of
$\mathfrak{A}$.   It is easy to check that $\LL_{\FS}$ is indeed an
order;  this is a consequence of item (1) of the previous paragraph.
We will also have use for a smaller subspace
$\LL_{\FS}^{\circ}\subset\LL_{\FS}$, defined by
$$\LL_{\FS}^{\circ}=\Delta(\gp_E)+\gp_E^{n+1}\times\gp_E^{n+1}+\mathbf{V}^{n+1}.$$

\begin{prop}
\label{linkingorders} The linking order $\LL_\FS$ has the following
properties:
\begin{enumerate}
\item The group $\LL_{\FS}^{\times}$ is normalized by $\Delta(E^{\times})$.
\item With respect to $\nu_{\FS}$, the annihilator of $\LL_{\FS}$ is $\LL_{\FS}^{\circ}$.
\item $\LL_{\FS}^{\circ}$ is a double-sided ideal of $\LL_{\FS}$.
\item If $E/F$ is ramified, then $\LL_{\FS}/\LL_{\FS}^{\circ}$ is a commutative ring of order $q^2$, isomorphic to $k[X]/(X^2)$.
\item If $E/F$ is unramified, then $\LL_{\FS}/\LL_{\FS}^{\circ}$ is a noncommutative ring of order $q^6$ whose isomorphism class depends only on $q$ (and not $n$).
\item $\LL_\FS^{\times}\cap \GL_2(F)=H_{\FS_1}$, and $\LL_\FS^{\times}\cap B^{\times}=H_{\FS_2}$.

\end{enumerate}
\end{prop}

\begin{proof}
Claim (1) is easy to check.   For claim (2), we calculate the
annihilator of $\LL_{\FS}$ as follows:
\begin{eqnarray*}
\LL_{\FS}^* &=& \left[\Delta(\OO_E)+\gp_E^n\times\gp_E^n+\mathbf{V}^n\right]^*\\
&=&
\Delta(\OO_E)^*\cap\left(\gp_E^n\times\gp_E^n\right)^*\cap\left(\mathbf{V}^n\right)^*
\end{eqnarray*}
The three terms to be intersected are
\begin{eqnarray*}
\Delta(\OO_E)^* &=&
\Delta(E)+\gp_E^{n+1}\times\gp_E^{n+1}+C_1\times C_2\circ,\text{ by Lemma~\ref{Operp}}\\
\left(\gp_E^n\times\gp_E^n\right)^* &=&
\gp_E\times\gp_E+C_1\times C_2^\circ\\
\left(\mathbf{V}^n\right)^*&=&(E\times E)\oplus
\mathbf{V}^{n+1},\text{ by Lemma~\ref{vvprop}}
\end{eqnarray*}
We claim the intersection is $\LL^{\circ}_{\FS}$. Indeed, for a pair
$(x_1,x_2)$ to lie in $\LL^*_{\FS}$, the first two equations imply $s_1(x_1),s_2(x_2)\in\gp_E$ and $s_1(x_1)\equiv
s_2(x_2)\pmod{\gp^{n+1}_E}$, and the third implies $(x_1-s_1(x_1),x_2-s_2(x_2))\in\mathbf{V}^{n+1}$.

Claim (3) follows from the inclusion
$\mathbf{V}^n\mathbf{V}^{n+1}\subset\gp_E^{n+1}\times\gp_E^{n+1}$,
which is easily checked.

For claims (4) and (5), let
$\mathcal{R}_{\FS}=\LL_{\FS}/\LL_{\FS}^{\circ}$.  Fix a uniformizer
$\pi_E$ of $E$.

In the case that $E/F$ is ramified, we have $\mathbf{V}^n=\mathbf{V}^{n+1}$, so there
is an isomorphism
$$\mathcal{R}_{\FS}\isom\frac{\Delta(\OO_E)+\gp^n_E\times\gp^n_E}{\Delta(\gp_E)+\gp_E^{n+1}\times\gp_E^{n+1}}.$$
The ``numerator" of the right-hand side is the ring of pairs
$(x,x+\pi_E^ny)\in \OO_E\times\OO_E$ with $x,y\in\OO_E$.  Define a
map
\begin{eqnarray*}
\mathcal{R}_{\FS}&\to& k\times k\\
(x,x+\pi_E^ny)&\mapsto& (\overline{x},\overline{y}),
\end{eqnarray*}
where if $z\in\OO_E$ we have put $\overline{z}=z\pmod{\gp_E}$.  It
is easily checked that this map is an isomorphism of (additive)
groups;  the multiplication law induced on $k\times k$ is
$(x_1,y_1)(x_2,y_2)=(x_1x_2,x_1y_2+x_2y_1)$, which is to say that
$\mathcal{R}_{\FS}\isom k[X]/(X^2)$.

Now suppose $E/F$ is unramified.  In this case $V=\mathbf{V}^n/\mathbf{V}^{n+1}$
is a vector space over $k_E$ of dimension 1.  We have
$\mathbf{V}^n\mathbf{V}^n\subset\gp_E^n\times\gp_E^n$.  On the other
hand the image of $\gp_E^n\times\gp_E^n$ in $\mathcal{R}_{\FS}$ may
be identified with $k_E$ via $(x_1,x_2)\mapsto
\overline{\pi_E^{-n}(x_1-x_2)}$.  For $v,w\in V$, let $v\cdot w$ be the
image of $vw\in \gp_E^n\times\gp_E^n$ under this latter map.  Then
$(v,w)\mapsto v\cdot w$ is a pairing $V\times V\to k_E$ which is
$k_E$-linear in the first variable and satisfies $w\cdot v=(v\cdot
w)^q$.  This pairing is nondegenerate by part (1) of
Lemma~\ref{vvprop}:  One of the factors of
$\mathbf{V}^n\mathbf{V}^n$ is always $\gp_E^n$.  Choose an isomorphism $\phi\from V\to k_E$ of $k_E$ vector spaces in such a way that $v\cdot w=\phi(v)\phi(w)^q$.

We are now ready to describe the ring $\mathcal{R}_\FS$:  let $R$ be the $k$-algebra of matrices $$[\alpha,\beta,\gamma]=\left(\begin{matrix} \alpha & \beta & \gamma  \\ & \alpha^q & \beta^q \\ & & \alpha\end{matrix}\right),$$  where $\alpha,\beta,\gamma\in k_E$.
Any element of $\LL_\FS$ is of the form $(x,x+\pi_E^n y)+v$, where
$x,y\in\OO_E$ and $v\in\mathbf{V}^n$. Define a map
\begin{eqnarray*}
\LL_\FS &\to& R\\
(x,x+\pi^n_E y)+v &\mapsto& [\overline{x}, \overline{y}, \phi(v)];
\end{eqnarray*}
it is easy to see that this map descends to a ring isomorphism $\mathcal{R}_\FS\to R$.  Therefore $\mathcal{R}_\FS$ is a
noncommutative ring of order $q^6$ whose isomorphism class is independent of $n$.

For claim (6), we begin with the fact that any element $b$ of
$\LL_\FS^{\times}$ is of the form $(x+\pi^n y,x)+v$, with $x\in\OO_E^{\times}$,
$y\in\OO_E$, and $v\in
\mathbf{V}^n=\mathbf{V}_1^n\times\mathbf{V}_2^n$. If such an element
has $B$-component 1 we must have $x=1$ and $v=(v_1,0)$, which is to
say that $b=(1+\pi^n y,1)+(v_1,0)\in
\left(1+\gp_E^n+V^n\right)\times \set{1}$ is an element of $H_\FS$.
The argument for $B^{\times}$ is similar.
\end{proof}

In the sequel, we will construct a representation $\rho_\FS$ of the
unit group $\LL_{\FS}^{\times}$ inflated from a representation of the
finite group $(\LL_{\FS}/\LL_{\FS}^{\circ})^{\times}$.  Then when
$\rho_\FS$ is extended to $\Delta(E^\times)(F^\times\times F^\times)\LL^\times$ and induced up to $\GL_2(F)\times B^{\times}$, the result will
realize the Jacquet-Langlands correspondence for representations of
$\GL_2(F)$ containing the stratum $\FS$.  For completeness' sake, we
also want to construct the correspondence for supercuspidal
representations of level 0.  To this end we define the linking order
of level 0 by
$$\LL_0=M_2(\OO_F)\times \OO_B$$
and its double-sided ideal by
$$\LL_0^{\circ}=\gp_FM_2(\OO_F)\times\gP_B.$$
Let $E$ be the unique unramified quadratic extension of $F$ and
choose embeddings $E\injects M_2(F)$, $E\injects B$ so that
$M_2(\OO_F)\cap E=\OO_B\cap E=\OO_E$.  Let $s_1\from M_2(\OO_F)\to
E$ and $s_2\from B\to E$ be the projections as in the previous
section, let $\nu$ be an additive character of $E$ vanishing on
$\gp_E$ but not on $\OO_E$, and let $\nu_0\from \mathbf{A}\to\C^{\times}$
be the character $\nu_0(x_1,y_1)=\nu(s_1(x_1)-s_2(y_1))$.  Then
Prop.~\ref{linkingorders} has the following analogue in level zero:
\begin{prop}
\label{linkingorderslevelzero}  The linking order $\LL_0$ has the
following properties:
\begin{enumerate}
\item $\LL_0^{\times}$ is normalized by $\Delta(E^{\times})$.
\item With respect to $\nu_0$, the annihilator of $\LL_0$ is
$\LL_0^\circ$.
\item $\LL_0/\LL_0^{\circ}\isom M_2(k_F)\times k_E$.
\item $\LL_0^{\times}\cap \GL_2(F)=\GL_2(\OO_F)$, and $\LL_0^{\times}\cap
B^{\times}=\OO_B^{\times}$.
\end{enumerate}
\end{prop}

\section{Representations of $\LL_{\FS}^{\times}$ and the Fourier
transform.}\label{repsofb}

Keep the notations from the previous section:  Let
$\FS=(\mathfrak{A}_1,n_1,\alpha_1)$ be a simple stratum in
$\GL_2(F)$, let $\FS'=(\mathfrak{A}_2,n_2,\alpha_2)$ be its
corresponding simple stratum in $B^{\times}$, let $n=n_1$, let $\LL_\FS$ be
the associated linking order, let $\mathcal{R}_\FS$ be its quotient
ring by the ideal $\LL_\FS^\circ$, and let $\nu_\FS$ be the
associated additive character on $\mathbf{A}=M_2(F)\times B$. Let
$\mathbf{G}=\GL_2(F)\times B^{\times}$. For $g=(g_1,g_2)\in\mathbf{G}$,
write $$\norm{g}=\abs{\det g_1}_F\abs{\N g_2}_F.$$

We let $\mu_{\FS}$ be the unique Haar measure on the additive group
$\mathbf{A}$ which is self-dual with respect to $\nu_{\FS}$, and let
$\CF_\FS$ be the Fourier transform with respect to $\psi_{\FS}$:
$$\CF_\FS f(y)=\int_{\mathbf{A}} f(x)\nu_\FS(xy)d\mu_\FS(x).$$
There are translation operators $L,R\from\mathbf{G}\to\Aut
C^\infty_c(\mathbf{G})$, defined by $L_gf(y)=f(g^{-1}y)$ and
$R_hf(y)=f(yh)$;  we have the rules
\begin{equation}
\label{FTproperty} L_g \CF_\FS = \norm{g}^{2}\CF_\FS R_g,\;\;R_h
\CF_\FS =\norm{h}^{-2}\CF_\FS L_h.
\end{equation}

Let $\mathcal{R}_\FS$ be the $k_E$-algebra
$\LL_{\FS}/\LL_{\FS}^{\circ}$ as in the proof of
Prop.~\ref{linkingorders}.

\begin{prop} \label{measure} The measure of $\LL_{\FS}^{\circ}$ with respect to $\mu_{\FS}$ is $\#\mathcal{R}_\FS^{-1/2}$.
\end{prop}
\begin{proof}  Let $\chi_{\LL_{\FS}}$ be the characteristic function of $\LL_{\FS}$.  Then
$$\CF_\FS\chi_{\LL_{\FS}}(y)=\int_{\LL_{\FS}} \nu_{\FS}(xy)\;d\mu_{\FS}(x)$$
is supported on $\LL_{\FS}^\perp =\LL_{\FS}^{\circ}$ and equals
$\mu_{\FS}(\LL_{\FS})$ there;  {\em i.e.}
$\CF_\FS\chi_{\LL_{\FS}}=\mu_{\FS}(\LL_{\FS})\chi_{\LL_{\FS}^{\circ}}$.
Similarly
$\CF_\FS^2\chi_{\LL_\FS}=\mu_{\FS}(\LL_{\FS})\mu_{\FS}(\LL_{\FS}^{\circ})\chi_{\LL_{\FS}}$.
On the other hand, since $\mu_{\FS}$ is self-dual, we must have
$\CF_\FS^2\chi_{\LL_{\FS}}=\chi_{\LL_{\FS}}$, implying
$\mu_{\FS}(\LL_{\FS})\mu_{\FS}(\LL_{\FS}^{\circ})=1$.  Since
$\mu_{\FS}(\LL_{\FS})=\# \mathcal{R}_\FS\mu_{\FS}(\LL_{\FS}^{\circ})$,
the result follows.\end{proof}

Let $\mathcal{C}(\mathcal{R}_\FS)$ be the space of complex-valued
functions on $\mathcal{R}_\FS$.  Note that the character $\nu_{\FS}$
vanishes on $\LL_{\FS}^{\circ}$ and therefore induces a well-defined
additive character of $\mathcal{R}_\FS$.   We identify
$\mathcal{C}(\mathcal{R}_\FS)$ with a subspace of
$C^\infty_c(\mathbf{A})$.

Prop.~\ref{measure} together with the key property that $\mathcal{L}_\FS$ and $\mathcal{L}_\FS^\circ$ are dual lattices imply the following:

\begin{prop}  The Fourier transform $f\mapsto\CF_\FS f$ preserves the space $\mathcal{C}(\mathcal{R}_\FS)$.
For $f\in\mathcal{C}(\mathcal{R}_\FS)$, we have
\begin{equation}
\label{DFT} \CF_\FS{f}(y)=\#\mathcal{R}_\FS^{-1/2}\sum_{x\in
\mathcal{R}_\FS} f(x)\nu_{\FS}(xy).
\end{equation}
\end{prop}


Recall that the data of $\FS$ and $\FS'$ determine characters
$\psi_{\alpha_1}$ and $\psi_{\alpha_2}$ of the subgroups
$U_{\mathfrak{A}_1}^{n_1}$ and $U_{\mathfrak{A}_2}^{n_2}$ of
$\mathfrak{A}_1^{\times}$ and $\mathfrak{A}_2^{\times}$, respectively.  The
product group $U_{\mathfrak{A}_1}^{n_1}\times
U_{\mathfrak{A}_2}^{n_2}=1+\gp_E^n\mathfrak{A}_1\times\gp_E^n\mathfrak{A}_2$
is a subgroup of $\LL_\FS^{\times}$, and the product character
$\psi_\FS=\psi_{\alpha_1}\times\psi_{\alpha_2}^{-1}$ vanishes on
$\left(U_{\mathfrak{A}_1}^{n_1}\times
U_{\mathfrak{A}_2}^{n_2}\right)\cap
(1+\LL_\FS^\circ)=U_{\mathfrak{A}_1}^{n+1}\times
U_{\mathfrak{A}_2}^{n+1}$. Therefore if we let $\mathbf{U}_\FS$ be
the image of $U_{\mathfrak{A}_1}^{n_1}\times
U_{\mathfrak{A}_2}^{n_2}$ in $\mathcal{R}_\FS$, then $\psi_\FS$
induces a well-defined nontrivial character of $\mathbf{U}_\FS$.

We are now ready to construct the special representation
$\rho_{\FS}$.  Its relevant properties are as follows:

\begin{Theorem}  \label{rhonu} There exists an irreducible representation $\rho_{\FS}$ of $\mathcal{R}_\FS^{\times}$ satisfying the conditions:
\begin{enumerate}
\item $\rho_{\FS}$ vanishes on $k^{\times}\subset \mathcal{R}_{\FS}^{\times}$.
\item $\rho_{\FS}\vert_{\mathbf{U}_\FS}$ is a sum of copies of $\psi_\FS$.
\item If $f\in\mathcal{C}(\rho_\FS)$ is a matrix coefficient, then $\CF_\FS
f$ is supported on $\mathcal{R}_\FS^{\times}$ and satisfies $\CF_\FS
f(y)=\pm f(y^{-1})$, all $y\in\mathcal{R}_\FS^{\times}$.  The sign is 1 if $E/F$ is ramified and $-1$ otherwise.
\end{enumerate}
\end{Theorem}

\begin{rmk}
These three properties correspond to the three desired properties of
the representation $\Pi_\FS$ listed at the beginning of
Section~\ref{linkingorderssection}.
\end{rmk}
\begin{proof}
First, consider the case where $E=F(\alpha)$ is a ramified extension
of $F$.  Then by Prop.~\ref{linkingorders} we have an isomorphism
$\mathcal{R}_\FS\isom k[X]/(X^2)$ with respect to which $\nu_\FS$ is a
nontrivial additive character which vanishes on $k\subset\mathcal{R}_\FS$.
The subgroup $\mathbf{U}_\FS\subset\mathcal{R}^{\times}_\FS$ corresponds to
$\set{1+aX\;\vert\; a\in k}$.   There is obviously a unique
character $\rho_\FS$ of $\mathcal{R}^{\times}_\FS$ lifting $\psi_\FS$ and vanishing
on $k^{\times}$.  It takes the form $$\rho_\FS(a+bX)=\Psi(a^{-1}b),$$ where
$\Psi\from k\to\C^{\times}$ is a nontrivial character determined by $\psi_\FS$.   That $\rho_\FS$ satisfies claim (3) is a simple calculation in the commutative ring $\mathcal{R}_\FS$.



The case of $e=1$ is far more subtle.  The required representation
$\rho_\FS$ is related to the construction of the Weil representation
of a symplectic group over a finite field.  We present a self-contained version of the construction in the following section.\end{proof}

\subsection{Fourier transforms on the Heisenberg group.}

In this section, $k$ is the finite field with $q$ elements and $k_2/k$ is a quadratic field extension.  As in the proof of Prop.~\ref{linkingorders}, let $R$ be the $k$-algebra of matrices of the form $$[\alpha,\beta,\gamma]=\left(\begin{matrix} \alpha & \beta & \gamma  \\ & \alpha^q & \beta^q \\ & & \alpha\end{matrix}\right),$$  where $\alpha,\beta,\gamma\in k_2$.   Let $U\subset R^\times$ be the subgroup of matrices of the form $[1,0,\gamma]$, and let $U^1\subset U$ be the subgroup consisting of those $[1,0,\gamma]$ for which $\tr_{k_2/k}\gamma=0$.   Note that the center of $R^\times$ is $k^\times U$.

Let $\ell$ be a prime not dividing $q$, and let $\nu_k\from k\to \overline{\Q}_\ell^\times$ be a nontrivial additive character.  Define an additive character $\nu_R$ of $R$ by $\nu_R([\alpha,\beta,\gamma])=\nu_k(\tr_{k_2/k}\gamma)$.  Let $\mathcal{F}$ be the Fourier transform with respect to $\nu_R$.

\begin{Theorem}\label{rhoconstruction}
For each character $\psi$ of $U$ which is nontrivial on $U^1$, there exists a representation $\rho_\psi$ of $R^\times$ satisfying the properties:
\begin{enumerate}
\item $\rho_\psi$ is trivial on $k^\times$.
\item $\rho_\psi\vert_U$ is a multiple of $\psi$.
\item For a matrix coefficient $f\in \mathcal{C}(\rho_\psi)$, the Fourier transform $\mathcal{F}f$ is supported on $R^\times$ and satisfies $\mathcal{F}f(y)=-f(y^{-1})$ for $y\in R^\times$.
\end{enumerate}
\end{Theorem}

The proof will occupy the rest of the section.   To construct $\rho_\psi$, we will build a nonsingular projective curve $X/\overline{k}$ admitting an action of $R^\times$, and find $\rho$ in the $\ell$-adic cohomology of $X$.

First, we recognize a relationship between $R^\times$ and the unitary group $\GU_3$.  Let $\Phi$ be the matrix
$$\Phi=\left(\begin{matrix}  & & 1\\ & -1 & \\ 1& &  \end{matrix}\right),$$
and let $\GU_3(k)$ be the subgroup of matrices $M\in \GL_3(k_2)$ satisfying $M^*\Phi M=\lambda(M)\Phi$ for a scalar $\lambda(M)$.  (Here $M^*$ is the conjugate transpose of $M$.)  Then a large part of the Borel subgroup of $\GU_3(k)$ is contained in $R^\times$.  Indeed, if $M\in R^\times$, we can measure the defect of $M$ from lying in $\GU_3(k)$ by a homomorphism $\delta\from R^\times\to k$ defined by
\begin{equation}
\Phi^{-1}M^*\Phi M = \lambda(M)\left(\begin{matrix} 1 & & \delta(M) \\ & 1 & \\ & & 1\end{matrix}\right).
\end{equation}
Explicitly, $\delta([\alpha,\beta,\gamma])=\alpha\gamma^q+\alpha^q\gamma-\beta^{q+1}$.
Let $R^1=\ker \delta$;  then $R^1\subset \GU_3(\F_q)$.

The algebraic group $\GU_3$ acts on the projective plane $\mathbf{P}^2_k$ in the usual manner;  the group $\GU_3(k)$ preserves the equation $y^{q+1}=x^qz+xz^q$ in projective coordinates.  This equation defines a nonsingular projective curve $X^1$ of genus $q(q-1)/2$ with an action of $\GU_3(k)$.  Let $X=R^\times \times_{R^1} X^1$;  this is a smooth projective curve with an action of $R^\times$.  Let $\ell$ be a prime distinct from the characteristic of $k$, and let $\rho\from R^\times\to H^1(X,\overline{\Q}_\ell)$ be the representation of $R^\times$ on the first cohomology of $X$.   The degree of $\rho$ is $q^2(q-1)$.    Note that $\rho$ is trivial on $k^\times\subset R^\times$.

Since $U$ lies in the center of $R^1$, we have a decomposition $\rho=\oplus_\psi \rho_\psi$ of $\rho$ into its irreducible $\psi$-isotypic components, where $\psi$ runs over characters of $U$ which are nontrivial on $U^1$;  each has dimension $q$.   We claim that $\rho_\psi$ is irreducible.  By the discrete Stone von-Neumann theorem there is a unique irreducible representation $\varsigma$ of the $p$-Sylow subgroup $H\subset R^\times$ which lies over $\psi$, and furthermore $\deg \varsigma=q$.  Since the restriction of $\rho_\psi$ to $H$ lies over $\psi$ and has degree $q$, it must agree with $\varsigma$.  Therefore $\rho_\psi$ is irreducible.

Let $T\subset R^\times$ be the subgroup of diagonal matrices, so that $T\isom k_2^*$.  The Lefshetz fixed-point theorem can easily be used to compute the restriction of $\rho_\psi$ to $T$:

\begin{prop}\label{multone1} The restriction of $\rho_\psi$ to $T$ is exactly the direct sum of those characters $\chi$ of $T$ which are nontrivial on $T/k^\times$.
\end{prop}

For a matrix coefficient $f\in\mathcal{C}(\rho_\psi)$, we consider the Fourier transform $\mathcal{F}f$.  We claim that the Fourier transform $\mathcal{F}f$ is supported on $R^\times$.  Indeed, if $y\in R$ is not invertible then $uy=y$ for all $u\in U$.  It follows from this that $\mathcal{F}f(uy)=\mathcal{F}f(uy)=\psi(u)^{-1}\mathcal{F}f(y)$ for all $u\in U$;  since $\psi$ is nontrivial we see that $\mathcal{F}f(y)=0$.

Next we claim that for $y\in R^\times$ we have 
\begin{equation}\label{HeisenbergFE}
\mathcal{F}f(y)=-f(y^{-1}).
\end{equation}
Formally, we have $\mathcal{F}f(y)=f(y^{-1})\mathcal{F}(1)$, so in fact is suffices to show that 
\begin{equation}\label{HeisenbergFE1}
\mathcal{F}f(1)=-f(1).
\end{equation}
It is enough to prove Eq.~\ref{HeisenbergFE1} in the case that $f$ equals the character of $\rho_\chi$.  This is because the character of $\rho_\psi$ generates $\mathcal{C}(\rho_\psi)$ as an $(R^\times\times R^\times)$-module, and because the property in Eq.~\ref{HeisenbergFE} is invariant when we replace $f$ by any of its $(R^\times\times R^\times)$-translates.   Therefore let $f=\tr\rho_\psi$ be the character of $\rho_\psi$.

We have 
$$\mathcal{F}f(1)=\frac{1}{q^3}\sum_{x\in R^\times}\tr \rho_\psi(x)\nu_R(x).$$
We observe that the term $\rho_\psi(x)\nu_R(x)$ only depends on the conjugacy class of $x$ in $R^\times$.  We first dispense with those terms in the above some for which $x$ has eigenvalues in $k^\times$.  The sum over these terms vanishes, because for such an $x$ we have $\tr\rho_\psi(xu)\nu_R(xu)=\psi(u)\tr\rho_\psi(x)\nu_R(x)$ for all $u\in U^1$.  All that remains are the elements $x=[\alpha,\beta,\gamma]$ with $\alpha\in k_2^\times\backslash k^\times$, and each of these are conjugate to a unique element of the form $tu$, with $t\in T\backslash k^\times$ and $u\in U$.  Each such conjugacy class has cardinality $q^2$, and the value of $\tr\rho_\psi(tu)$ on such a class is $-\psi(u)$.    Therefore
\begin{eqnarray*}
\mathcal{F}f(1)&=&-\frac{1}{q}\sum_{t\in T\backslash k^\times}\sum_{u\in U} \psi(u)\nu_r(tu).
\end{eqnarray*}
This reduces to $-q=-f(1)$ by a simple calculation, thus completing the proof of Theorem~\ref{rhoconstruction}.  

\begin{rmk}  The curve $X$ is isomorphic (over $\overline{k}$) to the Fermat curve $x^{q+1}+y^{q+1}+z^{q+1}=0$.  It appears in the construction of the so-called unipotent representation of $\GU_3(k)$;  see~\cite{Lusztig}.

There is also a connection to the theory of the discrete Weil representation.  We have $R^\times=T\rtimes H$, where $H$ is the $p$-Sylow subgroup of $R^\times$.  Furthermore, $U\cap H=U^1$ is the center of $H$.  Write $\psi^1$ for the (nontrivial) restriction of $\psi$ to $U^1$.  The group $H/\ker\psi^1$ is a discrete Heisenberg group.     By the Stone von-Neumann theorem, there is a unique irreducible representation $V_\psi$ of $H$ lying over $\psi$.  

The group $T$ embeds as a nonsplit torus in $\SL_2(k)$, and the conjugation action of $T$ on $H/\ker\psi^1$ extends to an action of $\SL_2(k)$ in a manner which fixes each element of $U^1$.  The uniqueness property of $V_\psi$ means that if $\alpha\in\SL_2(k)$ and $^\alpha V_\psi$ is the conjugate representation $g\mapsto V_\psi(\alpha(g))$, then there is an isomorphism $W(\alpha)\from ^\alpha V_\psi\isom V_\psi$ which is well-defined up to a scalar.  The operators $W(\alpha)$ give an {\em {a priori}} projective representation of $\SL_2(k)$ on the underlying space of $V_\psi$ which in fact lifts to a proper representation $W$, the Weil representation.  See for
instance~\cite{Gerardin:Weil} or, for the geometric point of view,
~\cite{Gurevich:Hadani:geometric}.  The operators $W(\alpha)$ together with the representation $V_\psi$ give a $q$-dimensional representation of $\SL_2(k)\rtimes H$;  restricting this to $T\rtimes H/\ker\psi^1=R^\times/\ker\psi^1$ gives the representation $\rho_\psi$ we have constructed in Theorem~\ref{rhoconstruction}. 

When $W$ is restricted to a nonsplit torus of $\SL_2(k)$, each nontrivial
character appears at most once, see Theorem 3
of~\cite{Gurevich:Hadani:fourier};  this implies the property of
$\rho_\psi$ given in Prop.~\ref{multone1}.
\end{rmk}

The case of $e=1$ in Theorem~\ref{rhonu}  follows from Theorem~\ref{rhoconstruction} once we observe the following:
\begin{enumerate}
\item There exists an isomorphism $\mathcal{R}_\FS\to R$.
\item Under this isomorphism, $\nu_\FS$ is identified with an additive character of the form $\nu_R$ described above.
\item The subgroup $\mathbf{U}_\FS\in\mathcal{R}_\FS^\times$ is identified with $U\subset R^\times$.
\item Choose an isomorphism $\iota\from\C\to\overline{\Q}_\ell$, then the complex character $\psi_\FS$ of $\mathbf{U}_\FS$ is identified with an $\ell$-adic character $\psi$ of $U$.
\item The condition that $\FS=(M_2(\OO_F),n,\alpha)$ be a simple stratum implies that the reduction of $\pi^n_F\alpha$ has irreducible characteristic polynomial, which in turn implies that $\psi$ is nontrivial on $U^1$.
\item The $\ell$-adic representation $\rho_\psi$ constructed in Theorem~\ref{rhoconstruction} with respect to the data of $\nu_R$ and $\psi$ may be transported via $\iota^{-1}$ to a complex representation of $\mathcal{R}_\FS^\times$ which satisfies the requirements of Theorem~\ref{rhonu}.
\end{enumerate}

\subsection{The case of level 0} \label{level0case} The linking order of level 0 is
$\LL_0=M_2(\OO_F)\times \OO_B$, and its quotient ring $\mathcal{R}_0$ is
$M_2(k)\times k_E$.  The additive character $\nu_0$ is of the form
$$\nu_0(x,y)=\nu(\tr_{M_2(k)/k} x -\tr_{k_E/k} y),$$ where $\nu$ is a nontrivial additive character of
$k$, and $\CF_0$ is the Fourier transform with respect to this
character. Let $\theta$ be a character of $k_E^{\times}$. Assume that
$\theta$ is {\em regular}, meaning that it does not factor through
the norm map $k_E^{\times}\to k^{\times}$.  It is well-known that there is an irreducible cuspidal
representation $\eta_\theta$ of $\GL_2(k_F)$ corresponding to $\theta$.  The character of this representation takes the value $-(\theta(\alpha)+\theta(\alpha^q))$ on an element $g\in\GL_2(k_F)$ with distinct eigenvalues $\alpha,\alpha^q\in k_E$ not lying in $k_F$.


Let $\rho_\theta$ be the character $\eta_\theta\otimes\theta^{-1}$
of $\mathcal{R}^{\times}_0=\GL_2(k_F)\times k_E^{\times}$.   The following proposition
concerns the Fourier transforms of matrix coefficients of
$\rho_\theta$.

\begin{prop}  \label{replevel0} For $f\in\mathcal{C}(\rho_\theta)$ we have
that $\CF_0 f$ is supported on $\mathcal{R}^{\times}_0$ and satisfies $\CF_0
f(y)=-f(y^{-1})$ for $y\in\mathcal{R}^{\times}_0$.
\end{prop}

\begin{proof}  We reduce this to two calculations relative to the
rings $M_2(k)$ and $k_E$, respectively.  Let $R_1=M_2(k)$,
$R_2=k_E$, and for $i=1,2$ let $\nu_i$ be the additive character of
$R_i$ defined by $\nu_i(x)=\nu_0(\tr_{R_i/k} x)$, so that
$\nu_\FS(x,y)=\nu_1(x)\nu_2(-y)$.



Write $\tau_{\theta,\nu}$ for the Gauss sum $\sum_{\alpha\in
k_E^{\times}}\theta(\alpha)\nu(\tr_{k_E/k_F}\alpha)$.  We claim that for all $f\in\mathcal{C}(\eta)$ we have that $\CF_1f$ is supported on $R_1^\times=\GL_2(k)$ and satisfies $$\CF_1f(y)=-\tau_{\theta,\nu}f(y^{-1}).$$

This is a straightforward calculation.  It is a special case of a calculation of epsilon factors of irreducible representations of $\GL_n$ which appears in~\cite{kondo};  these can always be expressed as a product of Gauss sums.  See also~\cite{macdonald}, Chap. IV.

The corresponding analysis for $R_2=k_E$ is simpler: define a
Fourier transform $\CF_2$ on $\mathcal{C}(R_2)$ by $\CF
f(y)=q^{-1}\sum_{x\in k_E^{\times}}f(x)\nu_2(-xy)$.  Then the Fourier
transform of the character $\theta^{-1}$ is supported on $k_E^{\times}$ and
equals $q^{-1}\tau_{\theta^{-1},\nu^{-1}}\theta$.

We may now complete the proof of the proposition.  For a
decomposable element $f=f_1\otimes f_2$ of $\mathcal{C}(R_1\times
R_2)$, we have $\CF_0 f=\CF_1 f_1\otimes \CF_2 f_2$.  If this same
$f$ is a matrix coefficient for
$\rho_\theta=\eta_\theta\otimes\theta^{-1}$ then we must have $\CF_0
f=-q^{-2}\tau_{\theta,\nu}\tau_{\theta^{-1},\nu^{-1}}f(y^{-1})$.  We
now use the classical identity of Gauss sums
$\tau_{\theta,\nu}\tau_{\theta^{-1},\nu^{-1}}=\#k_E=q^2$, and the
proof is complete.
\end{proof}

\section{Construction of the Jacquet-Langlands Correspondence}
\label{proof}

The construction of the family of rings $\LL_{\FS}$ together with
the representations $\rho_{\FS}$ of $\LL_{\FS}^{\times}$ will now be used
to construct certain representations $\Pi_{\FS}$ of $GL_2(F)\times
B^{\times}$.  We will then use Cor.~\ref{FEcor} to show that the family
$\Pi_{\FS}$ realizes the Jacquet-Langlands Correspondence.  This
will involve showing that the matrix coefficients of $\Pi_{\FS}$
satisfy the functional equation in Eq.~\ref{BIGFE} for sufficiently
many $\chi$. The heart of that calculation has already been
completed in Theorem~\ref{rhonu}.

Recall that $\mathbf{G}=\GL_2(F)\times B^{\times}$; this group has center
$Z(\mathbf{G})=F^{\times}\times F^{\times}$. Let $\FS=(\mathfrak{A},n,\alpha)$ be
a simple stratum in $M_2(F)$, and let
$\FS'=(\mathfrak{A}',n',\alpha')$ be its corresponding simple
stratum in $B$.  
From these data we have constructed a linking order $\LL_\FS$
and an irreducible representation $\rho_\FS$ of $\LL_\FS^{\times}$.  Let $\ell=n/e$, so that every supercuspidal representation of $\GL_2$ containing $\FS$ has level $\ell$, and likewise for $B^\times$.  
The intersection of $Z(\mathbf{G})$ with $\LL^{\times}_\FS$ is
$$Z(\mathbf{G})\cap \LL^{\times}_\FS=\set{(z_1,z_2)\in\OO_F^{\times}\times\OO_F^{\times}\;\vert\; v_F(z_1-z_2) \geq \ell}.$$ Here $v_F$ is the valuation on $F$.  By Theorem~\ref{rhonu}, $\rho_\FS$
vanishes on the diagonally embedded subgroup
$\Delta(F^{\times})\cap\LL_\FS^{\times}$. Choose a character $\omega$ of
$Z(\mathbf{G})$ which vanishes on $\Delta(F^{\times})$ and agrees with the
central character of $\rho_\FS$ on
$Z(\mathbf{G})\cap\LL_\FS^{\times}$.  We identify $\omega$ with a
character of $F^{\times}$ via its restriction to $F^{\times}\times\set{1}$.

We now extend $\rho_\FS$ to a representation on a larger group which contains $Z(\mathbf{G})$ and which intertwines $\rho_\FS$.  
Define a group $\mathcal{K}_\FS$ by $$\mathcal{K}_\FS=Z(\mathbf{G})\Delta(E^\times)\mathcal{L}_\FS^\times.$$  (Recall that $\Delta(E^\times)$ normalizes $\mathcal{L}_\FS^\times$, so this is indeed a group.)  There is a unique extension of $\rho_\FS$ to a representation $\rho_{\FS,\omega}$ of $\mathcal{K}_\FS$ which satisfies the conditions:
\begin{enumerate}
\item $\rho_{\FS,\omega}\vert_{Z(\mathbf{G})}=\omega$,
\item For $\beta\in E^\times$, $\rho_{\FS,\omega}(\Delta(\beta))=(-1)^{v_E(\beta)}\text{ if $E/F$ is ramified}$,
\item For $\beta\in E^\times$, $\rho_{\FS,\omega}(\Delta(\beta))=1\text{ if $E/F$ is unramified.}$
\end{enumerate}  

The group $\mathcal{K}_\FS$ is open and compact modulo its center.  We may now define the representation $\Pi_{\FS,\omega}$ of $\mathbf{G}$ as the induction of $\rho_{\FS,\omega}$ with compact supports:  $$\Pi_{\FS,\omega}=\Ind_{\mathcal{K}_\FS}^{\mathbf{G}}\rho_{\FS,\omega}.$$

We wish to confirm
that $\Pi_{\FS,\omega}$ satisfies the desired properties (1)-(3)
listed at the beginning of Section~\ref{linkingorderssection}.  It
is already apparent that (1) $\Pi_{\FS,\omega}$ vanishes on
$\Delta(F^{\times})$.  For property (2) we have the following:

\begin{Theorem}
\label{Ind}  $\Pi_{\FS,\omega}$ is the direct sum of representations of $\mathbf{G}$ of the form $\pi\otimes\check{\pi}'$, where $\pi$ (resp., $\pi'$) is a minimal supercuspidal irreducible representation of $\GL_2(F)$ (resp., $B^\times$) having central character $\omega$ and containing the stratum $\FS$ (resp., $\FS'$).  Every representation of either group having the above properties is contained in $\Pi_{\FS,\omega}$.  
\end{Theorem}

\begin{proof}  Note that $\mathcal{K}_\FS\subset J_\FS\times J_{\FS'}$ is a subgroup of finite index.   Let $$M=\Ind_{\mathcal{K}_S}^{J_\FS\times J_{\FS'}}\rho_{\FS,\omega}.$$  Then $M$ is a direct sum of irreducible representations of $J_{\FS}\times J_{\FS'}$ of the form $\Lambda\otimes\check{\Lambda}'$.   By Theorem~\ref{rhonu}, such a $\Lambda\otimes\check{\Lambda}'$ lies over the character $\psi_\FS=\psi_\alpha\otimes\psi_{\alpha'}^{-1}$ of $U_\FS\times U_{\FS'}$.   Therefore we have $\Lambda\in C(\psi_\alpha, \mathfrak{A})$ and $\Lambda'\in C(\psi_{\alpha'},\mathfrak{A}')$.  By Theorem~\ref{classification2}, $\pi=\Ind_{J_\FS}^{\GL_2(F)} \Lambda$ is an irreducible supercuspidal representation of $\GL_2(F)$ containing $\FS$.  Since $\rho_{\FS,\omega}$ has central character $\omega$, the same is true of $\pi$.  The reasoning is similar for $\pi'=\Ind_{J_\FS'}^{B^\times} \Lambda'$.  

Now assume $\pi$ is an irreducible supercuspidal representation of $\GL_2(F)$ containing $\FS$ with central character $\omega$.   We claim that $\pi$ is contained in $\Pi_{\FS,\omega}\vert_{\GL_2(F)}$.  Since $\Pi_{\FS,\omega}$ is induced from the representation $\rho_{\FS,\omega}$ of $\mathcal{K}_\FS$, the restriction of $\Pi_{\FS,\omega}$ to $\GL_2(F)$ contains $\Ind_{\mathcal{K}_\FS\cap\GL_2(F)}^{\GL_2(F)}
\rho_{\FS,\omega}$. Therefore to show that $\pi$ is contained in
$\Pi_{\FS,\omega}\vert_{\GL_2(F)}$ it suffices to prove that $\pi\vert_{\mathcal{K}_\FS\cap\GL_2(F)}$ meets
$\rho_{\FS,\omega}\vert_{\mathcal{K}_\FS\cap\GL_2(F)}$.   By
Prop.~\ref{linkingorders} we have
$$\mathcal{K}_\FS\cap\GL_2(F)=F^{\times}H_\FS.$$ The central characters of
$\pi$ and $\rho_{\FS,\omega}$ agree on $F^{\times}$ by hypothesis.
Therefore it suffices to show that $\pi\vert_{H_S}$ meets $\rho_S\vert_{H_S}$.  By Theorem~\ref{classification2}, $\pi$ contains a representation $\Lambda\in C(\mathfrak{A},\psi_\alpha)$.  This means that the restriction of $\pi$ to $H_S$ contains $\Lambda\vert_{H_S}$, which must agree with $\rho_S\vert_{H_S}$ by Theorem~\ref{HS}.  The case of a representation of $B^\times$ is similar.  
\end{proof}

The third required property of $\Pi_{\FS,\omega}$, concerning the
zeta functions attached to matrix coefficients of this
representation, shall follow from Prop.~\ref{rhonu}.  We will
start by translating Prop.~\ref{rhonu} into a statement
concerning the Fourier transforms of matrix coefficients of
$\Pi_{\FS,\omega}$. 

For a function $f$ on $\mathbf{G}$, and a real number $s$, let $f_s$ be the function $$f_s(g)=f(g)\norm{g_1}^{s-2}\norm{g_2}^{-s}.$$  If $f\in\mathcal{C}(\Pi_{\FS,\omega})$, we wish to consider Fourier transforms of the functions $f_s$.  The functions $f_s$ are supported on $\mathcal{K}_\FS$, which is not compact, so their Fourier transforms do not {\em a priori} converge.  Nonetheless we may formally define the Fourier transform $\widehat{f_s}$ by integrating $f_s(x)\psi_{\mathbf{A}}(xy)$ over each of the (compact) cosets of $\mathcal{L}_\FS^\times$ in $\mathbf{G}$.  Since $f$ is a linear combination of $\mathbf{G}\times\mathbf{G}$-translates of vectors in $\mathcal{C}(\rho_\FS)$, which are in turn supported on $\mathcal{L}_\FS^\times$, we see that the integral vanishes on all but finitely many of the cosets.  We now evaluate $\widehat{f_s}$.  

\begin{prop}
For a matrix coefficient $f\in\mathcal{C}(\Pi_{\FS,\omega})$, we have
\begin{equation}
\label{widehat}
\widehat{f_s}=-\check{f}_{2-s}.
\end{equation}
\end{prop}
  
\begin{proof}  We will first prove the corresponding statement relative to the Fourier transform $\mathcal{F}_S$:
\begin{equation}
\label{FStransform}
\mathcal{F}_Sf_s=\pm \check{f}_{2-s},
\end{equation}
where the sign is 1 if $E/F$ is ramified and $-1$ otherwise.  It will suffice to prove Eq.~\ref{FStransform} for matrix coefficients $f\in\mathcal{C}(\rho_S)$ supported on the group $\mathcal{L}_S^\times$.  Indeed, glancing at the rules in Eq.~\ref{FTproperty} shows that the validity of Eq.~\ref{FStransform} is unchanged upon replacing $f$ by $L_gR_hf$ for elements $g,h\in\mathbf{G}$, and these translates span $\mathcal{C}(\Pi_{S,\omega})$ as $f$ runs through $\mathcal{C}(\rho_S)$.  But for $f\in\mathcal{C}(\rho_S)$, Eq.~\ref{FStransform} follows from Theorem~\ref{rhonu}, because $f_s=f$.

To derive Eq.~\ref{widehat} from Eq.~\ref{FStransform} we must compare the Fourier transforms $\tilde{f}$ and $\mathcal{F}_\FS f$.  The first transform is taken relative to the additive character $\psi_{\mathbf{A}}$, while the second is taken relative to the character $\nu_S$.  The characters are related by $\nu_S(x)=\psi_{\mathbf{A}}(\Delta(\beta)^{-1}x)$ for an element $\beta\in E^\times$ of valuation $n$;  formally we have $\hat{f}=\norm{\Delta(\beta)}^{-1}R_{\beta}\mathcal{F}_Sf$.  Applying this to the function $f_s$, we see that 
\begin{eqnarray*}
\widehat{f_s}&=&\norm{\Delta(\beta)}^{-1}R_\beta\mathcal{F}_Sf_s\\
&=&\pm \norm{\Delta(\beta)}^{-1} R_\beta (\check{f})_{2-s}\\
&=&\pm (R_\beta\check{f})_{2-s},
\end{eqnarray*}
where the sign is positive if and only if $E/F$ is ramified.  If $E/F$ is ramified, then $\beta\in E^\times$ has odd valuation, and $R_\beta\check{f}=-\check{f}$ because $\rho_{\FS,\omega}$ takes the value $-1$ on such elements.  If $E/F$ is unramified, then $\rho_{\FS,\omega}(\Delta(\beta))=1$, and therefore $R_\beta\check{f}=\check{f}$.  The proposition follows.
\end{proof}

We are ready to prove the appropriate functional equation for the
zeta functions attached to $\Pi_{\FS,\omega}$.  Recall that for an
admissible representation $\Pi$ of $\mathbf{G}$, and for
$\Phi\in\C^\infty_c(\mathbf{A})$, $f\in\mathcal{C}(\Pi_{S,\omega})$, we defined
the zeta function 
\begin{eqnarray*}
\zeta(\Phi,f,s)&=&\int_{\mathbf{G}}
\Phi(g)f(g)\norm{g_1}^s\norm{g_2}^{2-s}\;d\mu^{\times}(g)\\
&=&\int_{\mathbf{G}}\Phi(g)f_s(g)\;d\mu(g)\\
\end{eqnarray*}
where $\mu$
is a Haar measure on $\mathbf{A}$.  
\begin{Theorem}\label{finalFE} For all $\Phi\in\mathcal{C}^\infty_c(\mathbf{A})$ all $f\in\mathcal{C}(\Pi_\omega)$, and all characters $\chi$ of $F^\times$ of conductor not exceeding $\ell$, we have $$\zeta(\Phi,\chi f,s)=-\zeta(\hat{\Phi},\chi^{-1}\check{f},2-s).$$
\end{Theorem}

\begin{proof} It suffices to prove the claim for $\chi=1$.  Indeed, if $f\in \mathcal{C}(\Pi_{S,\omega})$, then $\chi f$ lies in $\mathcal{C}(\Pi_{S',\chi^2\omega})$ for a different simple stratum $S'=(\mathcal{A}_1,n_1,\alpha'_1)$.  (Explicitly:  let $\beta\in\gp_E^{-n}$ be such that $(\chi\circ N_{E/F})(1+x)=\psi_F(\tr_{E/F}\beta x)$ for all $x\in\gp_E^n$;  then $\alpha_1'=\alpha_1+\beta$.)

Assume therefore that $\chi=1$.  We will take the measure $d\mu$ to equal $d\mu_\psi$, the measure dual to the character $\psi_{\mathbf{A}}$.  Since $\hat{\hat{\Phi}}(x)=\Phi(-x)$ we have that $\zeta(\hat{\hat{\Phi}},f,s)=\zeta(\Phi,f,s)$ by a change of variable $g\mapsto -g$ in the integral.  Now we apply Prop.~\ref{widehat}:
\begin{eqnarray*}
\zeta(\Phi,f,s)&=&\zeta(\hat{\hat{\Phi}},f,s)\\
&=&\int_{\mathbf{G}}\hat{\hat{\Phi}}(g)f_s(g)\;d\mu_\psi(g)\\
&=&\int_{\mathbf{G}}\hat{\Phi}(g)\hat{f_s}(g)\;d\mu_\psi(g)\\
&=&-\int_{\mathbf{G}}\hat{\Phi}(g)\check{f}_{2-s}(g)\;d\mu_\psi(g)\\
&=&-\zeta(\hat{\Phi},\check{f},2-s).
\end{eqnarray*}
\end{proof}

\subsection{The construction in level zero.}  The preceding constructions carry over easily to the case of level
zero. Let $E$ be an unramiifed quadratic extension of $F$. Letting
$\theta$ denote a regular character of $k_E^{\times}$, we constructed in
Section~\ref{level0case} a representation $\rho_\chi$ of the unit
group of the linking order $\LL_0$. Choose a central
character $\omega$ of $F^{\times}\times F^{\times}$ which agrees with the central
character of $\rho_\theta$ on $(F^{\times}\times F^{\times})\cap
\LL_0^{\times}=\OO_F^{\times}\times\OO_F^{\times}$.  Extend $\rho_\theta$ to a
representation $\rho_{\theta,\omega}$ of $\mathcal{K}_0=(F^{\times}\times
F^{\times})\LL_0$ agreeing with $\omega$ on the center.  Finally,
let $\Pi_{\theta,\omega}$ be the induced representation of
$\rho_{\theta,\omega}$ from $\mathcal{K}_0$ up to $\GL_2(F)\times
B^{\times}$.

Then Thm.~\ref{Ind} has the following analogue:
\begin{Theorem}\label{Indlevel0} Let $\pi$ be a minimal irreducible admissible
representation of $\GL_2(F)$ (resp., $B^{\times}$) with central character
$\omega$ (resp., $\omega^{-1}$).  The following are equivalent:
\begin{enumerate}
\item $\pi$ has level zero, and the restriction of $\pi$ to
$\GL_2(\OO_F)$ (resp., $\OO_B^{\times}$) contains a representation inflated
from the representation $\eta_\theta$ of $\GL_2(k)$ (resp., the
character $\theta$ of $k_E^{\times}$.)
\item $\pi$ is contained in $\Pi_{\theta,\omega}\vert_{\GL_2(F)}$ (resp.,
$\check{\pi}$ is contained in $\Pi_{\theta,\omega}\vert_{B^{\times}}$).
\end{enumerate}
\end{Theorem}

Similarly, Prop.~\ref{finalFE} has this analogue:
\begin{Theorem}
\label{finalFElevel0}  For $\Phi\in C^\infty_c(\mathbf{A})$,
$f\in\mathcal{C}(\Pi_{\omega,\theta})$, we have $$\zeta(\Phi,\chi
f,s)=-\zeta(\hat{\Phi},\chi^{-1}\check{f},2-s)$$ for all characters
$\chi$ of $F^{\times}$ which are trivial on $1+\gp_F$.
\end{Theorem}
The proofs of Thm.~\ref{Indlevel0} and Prop.~\ref{finalFElevel0} run
exactly the same as those of Thm.~\ref{Ind} and Prop.~\ref{finalFE}.

\subsection{Conclusion of the construction.}
Our construction of the Jacquet-Langlands correspondence is nearly complete.

\begin{Theorem}  For every irreducible representation $\pi'$ of $B^{\times}$ of dimension greater than one,
there is a supercuspidal representation $\pi$ of $\GL_2(F)$ for
which $\pi$ and $\pi'$ correspond.  Every supercuspidal
representation of $\GL_2(F)$ arises this way.
\end{Theorem}

\begin{proof} By Theorem~\ref{classification1}
we may twist $\pi'$ to assume either that $\check{\pi}'$ contains a
simple stratum $\FS'$, or else that it is level zero. In the first
case, Let $\FS=(\mathfrak{A},n,\alpha)$ be the corresponding stratum
in $M_2(F)$. Applying Theorem~\ref{Ind}, $\check{\pi}'$ is contained
in $\Pi_{\FS,\omega}\vert_{B^{\times}}$, where $\omega$ is the central
character of $\pi'$.   Suppose $\pi$ is a representation of
$\GL_2(F)$ appearing in $\Hom_{B^{\times}}(\check{\pi},\Pi_{\FS,\omega})$.
Then $\pi\otimes\check{\pi}'$ appears in $\Pi_{\FS,\omega}$.

Applying Theorem~\ref{Ind} again, we find that $\pi$ contains $\FS$.
Combining Cor.~\ref{FEcor} with Prop.~\ref{finalFE} shows that
$\pi'$ and $\pi$ correspond.

The logic is the same if $\pi'$ has level zero:  In this case
$\check{\pi}'$ contains a character of $\OO_B^{\times}$ inflated from a
character $\theta$ of a quadratic extension of $k$, so that
$\check{\pi}'$ is contained in $\Pi_{\theta,\omega}\vert_{B^{\times}}$.
Proceeding as above, we find a representation $\pi$ of $\GL_2(F)$
corresponding to $\pi'$.

If $\pi$ is a given supercuspidal representation of $\GL_2(F)$,
the argument above may be reversed to find a representation $\pi'$ of
$B^{\times}$ which corresponds to it. This concludes the proof.
\end{proof}

\bibliographystyle{amsalpha}
\bibliography{mybibfile}

\providecommand{\bysame}{\leavevmode\hbox to3em{\hrulefill}\thinspace}
\providecommand{\MR}{\relax\ifhmode\unskip\space\fi MR }
\providecommand{\MRhref}[2]{%
  \href{http://www.ams.org/mathscinet-getitem?mr=#1}{#2}
}
\providecommand{\href}[2]{#2}
\begin{thebibliography}{DKV84}

\bibitem[Bad02]{badulescu}
Alexandru~Ioan Badulescu, \emph{Correspondance de {J}acquet-{L}anglands pour
  les corps locaux de caract\'eristique non nulle}, Ann. Sci. \'Ecole Norm.
  Sup. (4) \textbf{35} (2002), no.~5, 695--747.

\bibitem[BH00]{HenniartJLII}
Colin~J. Bushnell and Guy Henniart, \emph{Correspondance de
  {J}acquet-{L}anglands explicite. {II}. {L}e cas de degr\'e \'egal \`a la
  caract\'eristique r\'esiduelle}, Manuscripta Math. \textbf{102} (2000),
  no.~2, 211--225.

\bibitem[BH05]{HenniartJLIII}
\bysame, \emph{Local tame lifting for {${\rm GL}(n)$}. {III}. {E}xplicit base
  change and {J}acquet-{L}anglands correspondence}, J. Reine Angew. Math.
  \textbf{580} (2005), 39--100.

\bibitem[BH06]{Henniart:Bushnell}
C.~Bushnell and G.~Henniart, \emph{The local langlands conjecture for
  $\text{GL}(2)$}, Springer-Verlag, 2006.

\bibitem[BK93]{BushnellKutzko}
Colin~J. Bushnell and Philip~C. Kutzko, \emph{The admissible dual of {${\rm
  GL}(N)$} via compact open subgroups}, Annals of Mathematics Studies, vol.
  129, Princeton University Press, Princeton, NJ, 1993.

\bibitem[BW04]{Wewers}
I.~Bouw and S.~Wewers, \emph{Stable reduction of modular curves}, Modular
  Curves and abelian varieties, Birkhauser, 2004.

\bibitem[Car86]{Carayol:ladicreps2}
H.~Carayol, \emph{Sur les repr\'esentations $\ell$-adiques associ\'ees aux
  formes modulaires de {H}ilbert}, Annales scientifiques de l'\'E.N.S.
  \textbf{19} (1986), no.~3, 409--468.

\bibitem[DKV84]{DKV}
P.~Deligne, D.~Kazhdan, and M.-F. Vign{\'e}ras, \emph{Repr\'esentations des
  alg\`ebres centrales simples {$p$}-adiques}, Representations of reductive
  groups over a local field, Travaux en Cours, Hermann, Paris, 1984,
  pp.~33--117.

\bibitem[G{\'e}r77]{Gerardin:Weil}
Paul G{\'e}rardin, \emph{Weil representations associated to finite fields}, J.
  Algebra \textbf{46} (1977), no.~1, 54--101.

\bibitem[G{\'e}r79]{Gerardin}
\bysame, \emph{Cuspidal unramified series for central simple algebras over
  local fields}, Automorphic forms, representations and {$L$}-functions
  ({P}roc. {S}ympos. {P}ure {M}ath., {O}regon {S}tate {U}niv., {C}orvallis,
  {O}re., 1977), {P}art 1, Proc. Sympos. Pure Math., XXXIII, Amer. Math. Soc.,
  Providence, R.I., 1979, pp.~157--169.

\bibitem[GH07]{Gurevich:Hadani:geometric}
S.~Gurevich and R.~Hadani, \emph{The geometric {W}eil representation}, Selecta
  Mathematica \textbf{13} (2007), no.~3, 465--481.

\bibitem[GH08]{Gurevich:Hadani:fourier}
S.~Gurevich and R.~Hadani., \emph{On the diagonalization of the discrete
  {F}ourier transform}, Applied and Computational Harmonic Analysis (2008).

\bibitem[GJ72]{Godement:Jacquet}
Roger Godement and Herv{\'e} Jacquet, \emph{Zeta functions of simple algebras},
  Lecture Notes in Mathematics, Vol. 260, Springer-Verlag, Berlin, 1972.

\bibitem[GL85]{GerardinLi}
Paul G{\'e}rardin and Wen-Ch'ing~Winnie Li, \emph{Fourier transforms of
  representations of quaternions}, J. Reine Angew. Math. \textbf{359} (1985),
  121--173.

\bibitem[Hen93]{HenniartJLI}
Guy Henniart, \emph{Correspondance de {J}acquet-{L}anglands explicite. {I}.
  {L}e cas mod\'er\'e de degr\'e premier}, S\'eminaire de {T}h\'eorie des
  {N}ombres, {P}aris, 1990--91, Progr. Math., vol. 108, Birkh\"auser Boston,
  Boston, MA, 1993, pp.~85--114.

\bibitem[How77]{Howe}
Roger~E. Howe, \emph{Tamely ramified supercuspidal representations of {${\rm
  Gl}\sb{n}$}}, Pacific J. Math. \textbf{73} (1977), no.~2, 437--460.

\bibitem[JL70]{JacquetLanglands}
Herv\'e Jacquet and Robert Langlands, \emph{Automorphic forms on {${\rm
  GL}(2)$}}, Lecture Notes in Mathematics, vol. 114, Springer-Verlag,
  Berlin-New York, 1970.

\bibitem[Kon63]{kondo}
Takeshi Kondo, \emph{On {G}aussian sums attached to the general linear groups
  over finite fields}, J. Math. Soc. Japan \textbf{15} (1963), 244--255.

\bibitem[Lus78]{Lusztig}
George Lusztig, \emph{Representations of finite {C}hevalley groups}, CBMS
  Regional Conference Series in Mathematics, vol.~39, American Mathematical
  Society, Providence, R.I., 1978, Expository lectures from the CBMS Regional
  Conference held at Madison, Wis., August 8--12, 1977.

\bibitem[Mac73]{macdonald}
I.G. Macdonald, \emph{Symmetric functions and {H}all polynomials}, Oxford
  University Press, 1973.

\bibitem[Rog83]{Rogawski}
Jonathan~D. Rogawski, \emph{Representations of {${\rm GL}(n)$} and division
  algebras over a {$p$}-adic field}, Duke Math. J. \textbf{50} (1983), no.~1,
  161--196.

\bibitem[Yos04]{yoshida}
Teruyoshi Yoshida, \emph{On non-abelian {L}ubin-{T}ate theory via vanishing
  cycles}, 2004.

\end{thebibliography}

\vspace{2cm}

\end{document}